\documentclass[letterpaper]{aspm}
\articleinfo{}{}{}

\usepackage{amssymb}
\usepackage{amsbsy}
\usepackage{amscd}
\usepackage{amsmath}
\usepackage{amsthm}
\usepackage{xypic}
\usepackage{enumerate}
\usepackage{epsfig}
\usepackage{rotating}
\def\ints{{\mathbb Z}}

\def\nats{{\mathbb N}}
\def\reals{{\mathbb R}}

\def\proj{{\mathbb P}}

\def\FF{{\mathbb F}}
\def\Gp{{\mathbb G}}

\def\del{{\partial}}

\def\qed{\hfill $\Box$}

\def\ord{{\text{ord}}}
\def\trace{{\text{tr}}}

\def\cf{{\text{char}}}
\def\Frac{{\text{Frac}}}

\def\coker{{\text{coker }}}

\DeclareMathOperator{\Aut}{Aut}

\DeclareMathOperator{\Spec}{Spec }
\def\Spf{{\mbox{Spf }}}

\def\mc#1{\mathcal{#1}}
\def\mf#1{\mathfrak{#1}}
\def\ol#1{\overline{#1}}

\newtheorem{theorem}{Theorem}[section]
\newtheorem{prop}[theorem]{Proposition}
\newtheorem{lemma}[theorem]{Lemma}
\newtheorem{conjecture}[theorem]{Conjecture}
\newtheorem{corollary}[theorem]{Corollary}
\newtheorem{predefinition}[theorem]{Definition}
\newenvironment{definition}{\begin{predefinition}\rm}{\end{predefinition}}
\newtheorem{preremark}[theorem]{Remark}
\newenvironment{remark}{\begin{preremark}\rm}{\end{preremark}}
\newtheorem{preconstruction}[theorem]{Construction}

\newtheorem{preframework}[theorem]{Framework}
\newenvironment{framework}{\begin{preframework}\rm}{\end{preframework}}
\newtheorem{prenotation}[theorem]{Notation}

\newtheorem{preexample}[theorem]{Example}
\newenvironment{example}{\begin{preexample}\rm}{\end{preexample}}
\newtheorem{preclaim}[theorem]{Claim}

\newtheorem{prequestion}[theorem]{Question}
\newenvironment{question}{\begin{prequestion}\rm}{\end{prequestion}}

\numberwithin{equation}{section}

\title{The (local) lifting problem for curves}

\author[Andrew Obus]{Andrew Obus}

\address{Columbia University Department of Mathematics. MC 4403, 2990 Broadway, New York, NY 10026 USA}

\email{andrewobus@gmail.com}

\rcvdate{Month Day, Year}
\rvsdate{Month Day, Year}

\subjclass[2000]{Primary 14H37, 12F10; Secondary 11G20, 12F15, 13B05, 13K05, 14G22, 14H30}

\keywords{branched cover, lifting, Galois group, Oort conjecture}
\thanks{The research was partially supported by an NSF Mathematical Sciences Postdoctoral Research Fellowship.  Revision of this work
took place at the Max-Planck-Institut F\"{u}r Mathematik in Bonn.  The author's travel to Kyoto was made possible by NSF Grant DMS 1044746.}

\begin{document}

\begin{abstract}
The \emph{lifting problem} that we consider asks: given a smooth curve in characteristic $p$ and a group of automorphisms, can
we lift the curve, along with the automorphisms, to characteristic zero?  One can reduce this to a local question (the so-called \emph{local
lifting problem}) involving continuous group actions on formal power series rings.
In this expository article, we overview much of the progress that has been made toward determining when the local lifting problem has a solution,
and we give a taste of the work currently being undertaken.  Of particular interest is the case when the group of automorphisms is 
cyclic.  In this case the lifting problem is expected to be solvable---this is the \emph{Oort conjecture}.
\end{abstract}

\maketitle
\tableofcontents

\section{Introduction}\label{Sintro}
Throughout this paper, $p$ represents a prime number, $k$ is an algebraically closed field of characteristic $p$, and $W(k)$ is
the ring of Witt vectors of $k$, that is, the unique complete discrete valuation ring (DVR) in characteristic zero with uniformizer $p$
whose residue field is $k$.  If $R/W(k)$ is a finite extension, then it will be assumed that $R$ is a DVR (i.e., $R$ is integral).  We note that any such
$R$ has residue field $k$.

\subsection{What is the lifting problem?}\label{Slifting}
Many problems in mathematics involve starting with some kind of data in characteristic $p$, and trying to ``lift" the data to characteristic zero (i.e.,
trying to find some analogous data in characteristic zero that reduce to the given data in characteristic $p$).  This paper considers the problem where
the data in question are a curve and a finite group of automorphisms.  More precisely, we have the following:

\begin{question}[Lifting problem]\label{Qlifting}
Let $Y$ be a smooth, proper curve over $k$.  Let $G$ be a finite group acting faithfully on $Y$ by $k$-automorphisms. 
Does there exist $R/W(k)$ finite, and a flat relative curve $Y_R \to \Spec R$ such that
\begin{enumerate}
\item $Y_R \times_R k \cong Y$, 
\item There is an action of $G$ on $Y_R$ by $R$-automorphisms such that the restriction of this action to the special fiber is the original 
$G$-action on $Y$?
\end{enumerate}
\end{question}
If the answers to (1) and (2) are ``yes," then we say that the $G$-action on $Y$ (or the curve $Y$ with $G$-action)
\emph{lifts to characteristic zero}  (or \emph{lifts over $R$}), and that $Y_R$ (with $G$-action) is a \emph{lift} of $Y$ (with $G$-action).

One might ask: does every $G$-action on every $k$-curve lift to characteristic zero?  In fact, the answer is ``no."  It is well known (see, e.g., 
\cite[IV, Ex.\ 2.5]{Ha:ag}) that the number of automorphisms of a curve of genus $g \geq 2$ in characteristic zero is at most $84(g-1)$.
But in characteristic $p$, a curve can have more automorphisms.  Roquette \cite{Ro:aa} 
gave the example of the smooth, projective model of the curve $Y$ given by
$y^2 = x^p - x$.  This curve has genus $(p-1)/2$ and $2p(p^2 - 1)$ automorphisms, which exceeds $84(\frac{p-1}{2} - 1)$ for 
$p \geq 5$. So clearly the canonical $\Aut(Y)$ action on $Y$ cannot lift to characteristic zero.

For another example, it is well known that the automorphism group of $\proj^1_L$ is $PGL_2(L)$ for any field $L$.  If $Y \cong \proj^1_k$, then
the group $G = (\ints/p)^n$ (for any $n$) is a subgroup of $PGL_2(k)$, and thus acts faithfully on $Y$.  
If this $G$-action on $Y$ lifts to characteristic zero, then
$G$ must act on $\proj^1_K$ for some characteristic zero field $K$, that is, $G \subseteq PGL_2(K)$.
But if $n > 1$ and $p^n \neq 4$, then $PGL_2(K)$ does not contain
$(\ints/p)^n$, so the $G$-action on $Y$ cannot lift to characteristic zero.

This of course leads one to ask: \emph{Can we find necessary and sufficient criteria for a $G$-action on $Y$ to lift to characteristic zero?
Over some particular $R$?} 

A group $G$ for which any $G$-action on any curve lifts to characteristic zero is called an \emph{Oort group} for $k$ (\cite{CGH:og}).

\subsection{What is the local lifting problem?}\label{Slocallifting}
Maintain the notation of Question \ref{Qlifting}.
It turns out that there is a local-global principle for lifting (Theorem \ref{Tlocalglobal}).  To wit, if we can lift 
the germs of the curve $Y$ at each of the points where $G$ acts with inertia to characteristic zero, along with the action of the inertia groups, 
then we can lift the $G$-action on $Y$ to characteristic zero.  Clearly, the converse holds as well.
So if we understand this \emph{local lifting problem}, we can understand the lifting problem.  More specifically, 
if $I_y \subseteq G$ is an inertia group at some $y \in Y$, then $I_y$ acts on the complete local ring of $y$ in $Y$, 
which is isomorphic to $k[[u]]$, by continuous $k$-automorphisms.  If there is a lift $Y_R$ of $Y$ with $G$-action to characteristic zero, 
then $I_y$ also acts on the complete local ring of $y$ in $Y_R$, which is isomorphic to $R[[U]]$.

In our description of the local lifting problem below, the letter $G$ can be thought of as one of these inertia groups $I_y$.  The action of $G$ on $k[[u]]$
is called a \emph{local $G$-action}.

\begin{question}[Local lifting problem]\label{Qlocallifting}
Let $G$ be a finite group, and suppose we have an embedding 
$\iota: G \hookrightarrow \Aut_{k, \text{cont}}(k[[u]])$.
Does there exist $R/W(k)$ finite, and an embedding $\iota_R: G \hookrightarrow \Aut_{R, \text{cont}}(R[[U]])$, such that, if $u$ is the
reduction of $U$, then the action of $G$ on $R[[U]]$ reduces to that of $G$ on $k[[u]]$?
\end{question}
If the answer is ``yes," then we say that the local $G$-action \emph{lifts to characteristic zero} (or \emph{lifts over $R$}), 
and the $G$-action on $R[[U]]$ is 
a \emph{lift} of the $G$-action on $k[[u]]$.  For a reformulation in terms of Galois theory, see Question \ref{Qgalois}.

Since the lifting problem does not always have a solution, neither does the local lifting problem.  So we ask:
 \emph{Can we find necessary and sufficient criteria for a local $G$-action to lift to characteristic zero?  Over some particular $R$?} 

A group $G$ for which every local $G$-action on $k[[u]]$ lifts to characteristic zero (and for which there exists a faithful local $G$-action)
is called a \emph{local Oort group} for $k$.

\subsection{Outline}\label{Soutline}
In \S\ref{Sglobal}, we give some basic results on the lifting problem and relate it to algebraic fundamental groups of curves.
In \S\ref{Slocalglobal} we reduce the lifting problem to the local lifting problem, which we focus on for the rest of the paper.  In \S\ref{Sgeneralities}, 
we rephrase the local
lifting problem in terms of Galois theory, and state some basic results on Galois extensions of $k[[t]]$.  Of particular importance is Lemma
\ref{Ljumpsdifferent}, which gives the differents of some such extensions in terms of their higher ramification filtrations.  We also state a weaker
form of the local lifting problem, the \emph{birational} local lifting problem, and sketch a proof that it can always be solved, unlike the 
(standard) local lifting problem.  Lastly, we discuss the relationship between solutions to the local lifting problem and solutions to
its birational variant.  

The heart of the paper begins with \S\ref{Sobstructions}, where we discuss obstructions to the existence of solutions to the local lifting 
problem.  The most important of these obstructions is the \emph{KGB obstruction} (\cite{CGH:ll}), for which we compute some examples in great detail.  
We also give some examples where the KGB obstruction vanishes and where local lifting either is possible or is thought to be possible.  The
remainder of the paper focuses on two interesting such examples.  
The case of lifting a \emph{cyclic} local $G$-action from characteristic $p$ to characteristic zero is the subject of \S\ref{Scyclic}.  
Such local actions are always expected to lift (this is the \emph{Oort conjecture}, see Conjecture \ref{Coort}), and have been shown to lift when
$p^3 \nmid |G|$ (\cite{OSS:ask} and \cite{GM:lg}).  We give an overview of the existing work in \S\ref{Spgroup}--\ref{Szp2lifting}, and in 
\S\ref{Szpnlifting}, we outline a new approach of Wewers and the author to the conjecture.  A success of this new approach has been to show that 
cyclic local $G$-actions lift to characteristic zero when $p^4 \nmid |G|$ (\cite{OW:oc}).  
In \S\ref{Smetacyclic}, we look at the case where $G \cong \ints/p \rtimes \ints/m$ is a nonabelian group.  This case is well-understood
due to the work of Bouw, Wewers, and Zapponi, and we give a summary of their work.  We also remark on the obstacles to generalizing to 
the case where $G$ has a larger $p$-Sylow subgroup.

Appendix \ref{Sgeometry} gives background on the non-archimedean geometry used in the paper, but is by no means a comprehensive reference.  The 
reader who has a basic familiarity with formal/rigid geometry should only need to refer to it for notational purposes, whereas the reader
without such familiarity will need to read it to have a reasonable understanding of much of \S\ref{Slocalglobal}, \S\ref{Scyclic}, and 
\S\ref{Smetacyclic}.  Appendix \ref{Sstablegraph} discusses the concepts of \emph{depth} and \emph{deformation data}, which may be unfamiliar
to most readers.  However, looking at this appendix may be safely postponed until the paper explicitly references it.

We note that the only other major exposition of the lifting and local lifting problems of which we are aware is Brewis's master's thesis 
(\cite{Br:ac}).  The aspects on which \cite{Br:ac} focuses are fairly complementary to those on which we focus here.  In particular, \cite{Br:ac}
gives many details on global techniques, the local-global principle, and the birational local lifting problem, whereas we spend the majority of our
time discussing specific examples of the local lifting problem.

\subsection{Notation and conventions}\label{Snotation} 
If $\Gamma$ is a group of automorphisms of a ring $A$, we write $A^{\Gamma}$ for the fixed ring under $\Gamma$.
For a finite group $G$, a \emph{$G$-Galois extension} (or \emph{$G$-extension}) of rings is a finite
extension $A \hookrightarrow B$ (also written $B/A$) of integrally closed integral domains such that the associated extension of fraction fields is 
$G$-Galois.  We do \emph{not} require $B/A$ to be \'{e}tale.
 
If $x$ is a scheme-theoretic point of a scheme $X$, then $\mc{O}_{X,x}$ is the local ring of $x$ in $X$.  
If $R$ is any local ring, then $\hat R$ is the completion of $R$ with respect to its maximal ideal. 
A \emph{branched cover} $f: Y \to X$ is a finite, generically \'{e}tale morphism of geometrically connected, normal schemes. 
A \emph{$G$-Galois cover} (or \emph{$G$-cover}) is a branched cover with an isomorphism $G \cong \Aut(Y/X)$ such that $G$ acts transitively on 
each geometric fiber of $f$.  Note that $G$-covers of affine schemes give rise to $G$-extensions of rings, and vice versa.

Suppose $f: Y \to X$ is a branched cover, with $X$ and $Y$ locally noetherian.  If $x \in X$ and $y \in Y$ are smooth codimension $1$ points such that 
$f(y) = x$, then the \emph{ramification index} of $y$ is the ramification index of the extension of complete local rings 
$\hat{\mc{O}}_{X, x} \to \hat{\mc{O}}_{Y, y}$.  If $f$ is Galois, then the \emph{branching index} of a smooth codimension $1$ point $x \in X$ 
is the ramification index of any point $y$ in the fiber of $f$ over $x$.  If $x \in X$ (resp.\ $y \in Y$) has branching index (resp.\
ramification index) greater than 1, then it is called a \emph{branch point} (resp.\ \emph{ramification point}).

If $R$ is any ring with a non-archimedean absolute value $| \cdot |$, then $R\{T\}$ 
is the ring of power series $\sum_{i=0}^{\infty} c_i T^i$ such that $\lim_{i \to \infty} |c_i| = 0$.  If $R$ has characteristic zero and 
residue characteristic $p$ (that is, if $0 < |p| < 1$), then we normalize the valuation on $R$ and on $K = \Frac(R)$ so that $p$ has valuation $1$.

If $X$ is a smooth curve over a complete discrete valuation field $K$ with valuation ring $R$, then a \emph{semistable}
model for $X$ is a relative curve $X_R \to \Spec R$ with $X_R \times_R K \cong X$ and semistable special fiber (i.e.,
the special fiber is reduced with only ordinary double points for singularities).  

Suppose $S$ is a ring of characteristic zero, with an ideal $I$ such that $S/I$ has characteristic $p$.  If an indeterminate in $S$ is given by a capital
letter, our convention (which we will no longer state explicitly) will be to write its reduction in $S/I$ using the respective lowercase letter.  
For example, if $I \subseteq W(k)[[U]]$ is the ideal generated by $p$, then $W(k)[[U]]/I \cong k[[u]]$, and $u$ is the reduction of $U$. 

If $B/A$ is a finite extension of Dedekind rings, the \emph{degree of the different} of $B/A$ is its length as a $B$-module.  If $L/K$ is a finite extension
of discrete valuation fields with valuation rings $A \subseteq K$ and $B \subseteq L$, 
then the degree of the different of $L/K$ is the degree of the different of $B/A$.

The group $D_n$ is the dihedral group of order $2n$.

\section{Preliminary global results}\label{Sglobal}
Before asking about lifting curves with automorphisms with characteristic zero, it makes sense to ensure that the curves themselves lift to 
characteristic zero, without worrying about automorphisms.  Luckily, this is true:

\begin{prop}\label{Ptrivgroup}
The trivial group is an Oort group.  Furthermore, the lifting of any curve with trivial action can be done over $W(k)$.
\end{prop}

\begin{proof}
Since $Y$ is a curve, the cohomology $H^2(Y, \mc{T}_Y)$ is trivial, where $\mc{T}_Y$ is the tangent sheaf of $Y$.  
By \cite[III, Corollaire 6.10]{sga1}, there is a smooth formal relative curve $\mc{Y}/W(k)$, whose special fiber is $Y$.  Since, in addition, 
$H^2(Y, \mc{O}_Y)$
is trivial, \cite[III, Prop.\ 7.2]{sga1} shows that $\mc{Y}$ is in fact the formal completion of an \emph{algebraic} relative curve $Y_{W(k)}/W(k)$ at the 
special fiber.  Then $Y_{W(k)}$ is the lift we seek. 
\end{proof}

As the next proposition shows, Proposition \ref{Ptrivgroup} 
allows us to shift our perspective from lifting \emph{curves} to lifting \emph{branched covers}.  Let $G$ act faithfully on $Y$, let $X \cong Y/G$,
and let $f: Y \to X$ be the canonical map.  Suppose $R/W(k)$ is finite.  Then lifting the $G$-action on $Y$ over $R$ is the same as lifting the $G$-cover 
$f$ to a $G$-cover $f_R: Y_R \to X_R$, where $Y_R$ and $X_R$ are lifts of the curves $Y$ and $X$ over $R$  
(see \S\ref{Snotation} for the definition of $G$-cover).

\begin{prop}\label{Ptame}
\begin{enumerate}[(i)]
\item If $G$ acts on $Y$ with trivial inertia groups, then the $G$-action on $Y$ lifts over $W(k)$.
\item In fact, even if $G$ acts on $Y$ with prime-to-$p$ inertia groups, then the $G$-action on $Y$ lifts over $W(k)$.  
In particular, any group of prime-to-$p$ order is an Oort group.
\end{enumerate}
\end{prop}

\begin{proof}
In case (i), the $G$-cover $f: Y \to Y/G \cong X$ is \'{e}tale.  By Proposition \ref{Ptrivgroup}, we can lift $X$ to a flat, smooth curve $X_{W(k)}$.
By Grothendieck's theory of \'{e}tale lifting (\cite[I, Corollaire 8.4]{sga1}, combined with \cite[III, Prop.\ 7.2]{sga1}), 
the $G$-cover $f$ lifts to a $G$-cover $f_{W(k)}: Y_{W(k)} \to X_{W(k)}$ over $W(k)$.
In case (ii), we instead use Grothendieck's theory of tame lifting (\cite[XIII, Corollaire 2.12]{sga1}, or \cite{We:dt} for an exposition).
\end{proof}

\begin{remark}\label{Rlocaltamelifting}
In \S\ref{Scyclic}, we will see how Proposition \ref{Ptame} also follows from the local-global principle.
\end{remark}

\begin{remark}\label{Rtamemotivation}
In fact, if $f: Y \to X$ is tamely ramified (as in the proof of Proposition \ref{Ptame}), 
then the generic fiber of $f_{W(k)}$ is branched at the same number of points with the same branching indices as $f$.
Furthermore, the lifting of $f$ over $W(k)$ (or over any finite $R/W(k)$) 
is unique once the branch points of the generic fiber are specified.  

This has the following consequence about
fundamental groups of curves:  Let $U$ be an affine curve over $k$, let $X$ be the smooth projective closure of $U$, let
$X_{W(k)}$ be a lift of $X$ to $W(k)$, and let $U_{W(k)} \subseteq X_{W(k)}$ be a lift of $U$ to $W(k)$ such that the generic fiber of
$X_{W(k)} \backslash U_{W(k)}$ contains exactly one point specializing to each point of $X \backslash U$.  Let $L$ be an algebraic closure of
$\Frac(W(k))$ and let $U_L$ and $X_L$ be the base changes of $U_{W(k)}$ and $X_{W(k)}$ to $L$.
If $\pi_1^t(U)$ is the \emph{tame fundamental
group} of $U$ (that is, the automorphism group of the pro-universal tame cover of $X$, \'{e}tale above $U$), 
and $\pi_1^t(U_L)$ is the \emph{$p$-tame fundamental group} of $U_L$ (same definition, but replace
``tame" with ``ramified of prime-to-$p$ index"), then there is a natural surjection $$\phi: \pi_1^t(U_L) \to \pi_1^t(U),$$ given as follows: If 
$\gamma \in \pi_1^t(U_L)$, then $\phi(\gamma)$ is determined by how it acts on tame covers $f: Y \to X$, \'{e}tale over $U$.  
Take the unique lift $f_{W(k)}$ of $f$ over $W(k)$ such that if $f_{\mc{O}_L}$ is the base change of $f_{W(k)}$ to the valuation ring $\mc{O}_L$ of $L$, 
the branch locus of the generic fiber of $f_{\mc{O}_L}$ 
is contained in $X_L \backslash U_L$.  Note that the special fiber of $f_{\mc{O}_L}$ is identical to that of 
$f_{W(k)}$.  By definition, $\gamma$ acts on the generic fiber of this base change of $f$, and 
the reduction of this action is the desired action on $Y$.
It is also clear that $\phi$ is surjective, as one can compatibly lift any compatible system of automorphisms of tame covers of $X$, \'{e}tale
over $U$.  Since $\pi_1^t(U_L)$ is well understood (as are all fundamental groups of curves over algebraically closed fields of characteristic 
zero), the existence of $\phi$ gives us useful information about $\pi_1^t(U)$.  For instance, we obtain that $\pi_1^t(U)$ is topologically finitely generated.
\end{remark}

Given that we can (essentially uniquely) lift tame covers of curves to characteristic zero, the next natural question is to
ask when we can lift wild covers, and what kinds of moduli these lifts have.  Hopefully, for affine curves $U/k$, this can shed some light 
on larger quotients of $\pi_1(U)$ than $\pi_1^t(U)$.  
In any case, this gives us motivation to study the lifting problem for a $G$-action on $Y$ by examining the local properties of the $G$-cover 
$f: Y \to Y/G$ near its ramification points.

\section{Reduction to the local lifting problem}\label{Slocalglobal}
The \emph{local-global principle} below allows us to study the lifting problem by way of the local lifting problem.

\begin{theorem}[Local-global principle]\label{Tlocalglobal}
Let $Y$ be a smooth, projective, connected curve over $k$, with a faithful action of $G$ by $k$-automorphisms.
Let $y_1, \ldots, y_s \in Y$ be the points where $G$ acts with nontrivial inertia.  For each $j$, $1 \leq j \leq s$,
let $G_j$ be the inertia group of $y_j$ in $G$, and let $\iota_j: G_j \hookrightarrow \Aut_{k, \text{cont}} k[[u_j]]$ be the induced local action 
on the complete 
local ring of $y_j$.  If $R/W(k)$ is finite, then the curve $Y$ (with $G$-action) lifts over $R$ iff each of the local $G_j$-actions lifts over $R$.
\end{theorem}

\begin{proof}[Sketch of proof] (see \cite[\S1.2]{Sa:fl}, or \cite[\S3]{Ga:pr} for many more details) 
Clearly, if the $G$-action on $Y$ lifts over $R$, then so do all the local actions.  Now, assume each of the local actions
lifts over $R$.  Consider the $G$-cover $f: Y \to X := Y/G$.  By Proposition \ref{Ptrivgroup}, $X$ has a smooth lift $X_R$ over $R$.  Let $\mc{X}$ 
be the formal completion of $X_R$ at $X$.
Let $B$ be the branch locus of $f$, set $V = X \backslash B$, and set $W = f^{-1}(V) = Y \backslash \{y_1, \ldots, y_s\}$.  Let $\mc{V} \subseteq \mc{X}$ 
be the formal subscheme associated to $V \subseteq X$ (\S\ref{Sformal}).
By Grothendieck's theory of \'{e}tale lifting (\cite[I, Corollaire 8.4]{sga1}), 
the $G$-cover $f|_{W}: W \to V$ lifts to an \'{e}tale $G$-cover of formal schemes $\mc{W} \to \mc{V}$ over $R$.  The boundary of $\mc{W}$ 
is isomorphic to disjoint union $\coprod_{j=1}^s \mc{B}_j$, where each $\mc{B}_j$ is isomorphic to the boundary of a disc (\S\ref{Sdiscs}).  
Each $\mc{B}_j$ corresponds to the point $y_j$, and $G_j$ acts on $\mc{B}_j$.

By assumption, each local $G_j$-action on $\hat{\mc{O}}_{Y, y_j}$
lifts over $R$ to a continuous $R$-linear $G_j$-action on the open disc $\mc{D}_j \cong \Spf R[[U_j]]$.  
The action of $G_j$ on $\mc{D}_j$ induces an action on its boundary $\partial \mc{D}_j$.  In fact, the $G_j$-action on $\partial \mc{D}_j$ is 
isomorphic to the action on $\mc{B}_j$.
Thus by identifying $\mc{B}_j$ and $\partial \mc{D}_j$, we can use formal patching to ``glue" each of these discs $\mc{D}_j$ to $\mc{W}$ in a 
$G_j$-equivariant way.  This yields a formal curve with $G$-action and projective special fiber 
(see, e.g., \cite[\S1.2]{Sa:fl}, or \cite[Ch.\ 3, \S4]{He:ht} for more details).  
By Grothendieck's Existence Theorem, this formal curve is the projective completion of a
smooth projective curve $Y_R$ with $G$-action.  This is the lift we seek.
\end{proof} 

\begin{remark}\label{Rbertinmezard}
Bertin and M\'{e}zard gave an alternate proof of Theorem \ref{Tlocalglobal} that relies on deformation theory (see \cite[\S3]{BM:df}).
There is also a proof by Green and Matignon (\cite[III]{GM:lg}).
\end{remark}

Because we have Theorem \ref{Tlocalglobal} at our disposal, the rest of this paper will focus exclusively on the local lifting problem.  For
interesting work involving global lifting, see \cite{Oo:ac}, \cite{OSS:ask}, and \cite{CGH:og}. 
For an overview of some of the techniques used in \cite{Oo:ac} and \cite{OSS:ask}, see \cite{Br:ac}, especially Chapter 3.

\section{Generalities on local extensions}\label{Sgeneralities}
\subsection{Power series rings}\label{Spowerseries}
Consider the power series ring $A = k[[u]]$.  A continuous $k$-automorphism $\gamma$ of $A$ (for the $u$-adic topology) is determined by the
image of $u$.  In particular, it is necessary and sufficient to have $\gamma(u) = \sum_{i=1}^{\infty} a_iu^i$, with
$a_1 \ne 0$.  A first, na\"{i}ve approach to the local lifting problem would be to try to write down a group $G$ of automorphisms of $A$ explicitly as 
power series, and then try to lift them explicitly.  However, this method is of limited usefulness.  It quickly becomes difficult even to write down 
power series of finite order in $k[[u]]$, let alone to write down a lift to $R[[U]]$, where $R/W(k)$ is finite.  For order $m$ prime to $p$, the 
series $\gamma(u) = \zeta_m u$, with $\zeta_m$ a primitive $m$th root of unity, works.  For order $p$, the series $\gamma(u) = \frac{u}{1+u} =
u - u^2 + u^3 - \cdots$
works.  But it is already not easy to see how to lift this to an order $p$ automorphism of 
$R[[U]]$ (note that lifting the power series $\gamma(u)$ to $R[[U]]$ na\"{i}vely will not in general result in an automorphism of order $p$).  

For automorphisms of order divisible by $p^2$, the situation is even worse.  In fact, the only explicit continuous 
$k$-automorphism of $k[[u]]$ of order $p^n$ known to the author, for $n \geq 2$, is an automorphism $\gamma$ of order $4$ given in \cite{CS:no} by
$$\gamma(u) = u + u^2 + \sum_{j=0}^{\infty}\sum_{\ell = 0}^{2^j-1}u^{6\cdot 2^j+ 2\ell}.$$
Trying to lift this automorphism while maintaining its order is a nightmare that we shall not attempt.  

In order to lift a $G$-action on $k[[u]]$, we must not only lift automorphisms with the correct orders, but we must also take into account the group
structure of $G$.  Since even lifting the automorphisms with the correct orders seems to be beyond our reach, we will not pursue this method further.

\subsection{Galois extensions}\label{Sgalois}
The more fruitful approach to the local lifting problem has been to view it from the perspective of Galois theory.  This is analogous to studying
the (global) lifting problem by studying branched covers of curves.
Suppose the group $G$ acts on $k[[u]]$ by continuous $k$-automorphisms.  Then $k[[u]]^G$ is a complete DVR with
residue field $k$, so by the structure theorem for complete DVRs (\cite[II, Theorem 2]{Se:lf}), it must be abstractly isomorphic to $k[[t]]$.  
Choosing an isomorphism gives $k[[u]]/k[[t]]$ the structure of a $G$-extension.  Conversely, if $A/k[[t]]$ is any $G$-Galois extension, 
then $A \cong k[[u]]$, again by the structure theorem.  Likewise, if $R/W(k)$ is finite, then
$R[[U]]^G \cong R[[T]]$ if $G$ acts by continuous $R$-automorphisms (\cite[Proposition 2.3.1]{Ra:sp}), 
and any $G$-extension $A_R/R[[T]]$ for which $A_R \otimes_R k \cong k[[u]]$ must satisfy
$A_R \cong R[[U]]$. 
We can now rephrase the local lifting problem as follows:
\begin{question}[Local lifting problem, Galois formulation]\label{Qgalois}
Suppose $G$ is a finite group, and $A/k[[t]]$ is a $G$-Galois extension, for a finite group $G$.  Does there exist $R/W(k)$ finite,
and a $G$-Galois extension $A_R/R[[T]]$ such that $A_R \otimes_R k \cong A$ and the $G$-action on $A_R$ reduces to the given $G$-action on $A$?
\end{question}
If such a lift exists, we say that $A_R/R[[T]]$ is a \emph{lift} of $A/k[[t]]$ over $R$, or that $A/k[[t]]$ \emph{lifts to characteristic zero} (or
\emph{lifts over $R$}).\\

\subsubsection{Properties of Galois extensions of complete DVRs.}\label{Sdvrextensions}
The following facts are from \cite[IV]{Se:lf}.
Let $F = k((t))$.  If $L/F$ is a finite $G$-Galois extension of complete discrete valuation fields, then 
$G$ must be of the form $P \rtimes \ints/m$, where $P$ is a $p$-group and $m$ is prime to $p$.  The group $G$ has a filtration $G 
= G_0 \supseteq G_i$ ($i \in \reals_{\geq 0}$) for the lower numbering, and $G
\supseteq G^i$ for the upper numbering ($i \in \reals_{\geq 0}$).  If $i \leq j$, then $G_i \supseteq G_j$ and $G^i
\supseteq G^j$.  
The subgroup $G_i$ (resp.\ $G^i$) is known as the \emph{$i$th higher ramification group for the lower numbering (resp.\ the upper 
numbering)}.  The two filtrations are related by Herbrand's formula (see \cite[IV, \S1, \S3]{Se:lf}).   

One knows that $G_0 = G^0 = G$, and that $G_1 = G^{\frac{1}{m}} = P$.   
For sufficiently large $i$, $G_i = G^i = \{id\}$.   Any $i$ such that $G^i \supsetneq G^{i + \epsilon}$ for all $\epsilon > 0$ is
called an \emph{upper jump} of the extension $L/F$.  Likewise, if $G_i \supsetneq G_{i+\epsilon}$ for $\epsilon > 0$, then $i$ is called a
\emph{lower jump} of $L/F$.  If $i$ is a lower (resp.\ upper) jump, $i > 0$, and $\epsilon > 0$ is sufficiently small,
then $G_i/G_{i + \epsilon}$ (resp.\ $G^i/G^{i + \epsilon}$) is an elementary abelian $p$-group.  The lower jumps are all integers.  

The above discussion applies equally to $G$-extensions of $k[[t]]$.  In particular, any group $G$ that arises in the local lifting 
problem is of the form $P \rtimes \ints/m$, where $P$ is a $p$-group and $p \nmid m$.  The degree $\delta$ 
of the different of a $G$-extension $A/k[[t]]$ is given by the formula
\begin{equation}\label{Ebasicdifferent}
\delta = \sum_{i=0}^{\infty} (|G_i| - 1).
\end{equation}  
The fact that 
the different can be determined from the higher ramification filtration for the lower numbering (and,
indeed, for the upper numbering also) will be important for determining whether a given extension 
of $R[[T]]$ for $R/W(k)$ finite is, in fact, a lift (see \S\ref{Sbirational}).\\

\begin{remark}\label{Rwildhigher}
Note that if $A/k[[t]]$ is a wildly ramified $G$-extension, then the degree of the different of $A/k[[t]]$ is always strictly greater than $|G| - 1$.  By the 
Hurwitz formula, this means that if $Y \to X$ is a wildly ramified cover of curves over $k$, then the genus of $Y$ is always higher than it would be 
if the cover had the same ramification points and indices, but was in characteristic $0$.
\end{remark}

\subsubsection{Cyclic Extensions.}\label{Scyclicjumps}
If $G \cong \ints/p^n$, then any $G$-extension must have $n$ different upper jumps $u_1 < \cdots < u_n$.  
\begin{lemma}\label{Ljumpsdifferent}
Set $u_0 = 0$.  If $L/K$ is a $\ints/p^n$-extension with upper jumps $u_1 < \cdots < u_n$, then the degree of the different of $L/K$ is 
$$p^n - 1 + \sum_{i=1}^n p^{i-1}(p^{n-i+1} - 1)(u_i - u_{i-1}).$$  
\end{lemma}

\begin{proof} This follows from (\ref{Ebasicdifferent}) and Herbrand's formula. \end{proof}

\subsection{Birational lifts and the different criterion}\label{Sbirational}
Usually, when dealing with Galois extensions of $k[[t]]$, it will be more convenient to deal with extensions of fraction fields than extensions of rings.
For instance, by Artin-Schreier theory, one knows that any $\ints/p$-extension $L/k((t))$ is given by an equation of the form $y^p - y = f(t)$.  
But writing down equations for the
integral closure of $k[[t]]$ in $L$ is much more difficult.  So we will often want to think of a Galois ring extension in terms of the associated extension
of fraction fields.  In particular, we have the following \emph{birational local lifting problem}:

\begin{question}[Birational local lifting problem]\label{Qbirational}
Let $A/k[[t]]$ be a $G$-extension.  Does there exist $R/W(k)$ finite, and a $G$-extension
$M/\Frac(R[[T]])$ such that:
\begin{enumerate}
\item If $A_R$ is the integral closure of $R[[T]]$ in $M$, then the integral closure of $A_R \otimes_R k$ is isomorphic to $A$,
\item The $G$-action on $\Frac(A) = \Frac(A_R \otimes_R k)$ induced from that on $A_R$ restricts to the given $G$-action on $A$?
\end{enumerate}
\end{question}
If so, we say that $A/k[[t]]$ \emph{lifts birationally to characteristic zero} and that
$A_R/R[[T]]$ is a \emph{birational lift} of $A/k[[t]]$ to characteristic zero.  

The birational local lifting problem evidently requires less than the local lifting problem to answer in the affirmative.  In fact, Garuti has 
proved the following:

\begin{theorem}[\cite{Ga:pr}]\label{Tbirational}
Any $G$-extension $A/k[[t]]$ lifts birationally to characteristic zero.
\end{theorem}

\begin{proof}[Sketch of proof]
If $A/k[[t]]$ is a $G$-extension, then by \cite[Theorem 1.4.1]{Ka:lg}, there is a $G$-cover $f: Y \to X := \proj^1_k$ such that $f$ is totally
ramified at a point $y \in Y$, and tamely ramified away from $y$.  The extension of complete local rings associated to $y \mapsto f(y)$ is isomorphic to $A/k[[t]]$.  Write
$U = X \backslash \{f(y)\}$ and $V = Y \backslash \{y\}$.   
Let $\mc{X}$ be the formal completion of $\proj^1_{W(k)}$ at its special fiber $X$, let $\mc{U}$ be the natural formal lift of $U$ inside $\mc{X}$
(see \S\ref{Sformal}), and let $\mc{D}$ be the formal open disc $\mc{X} \backslash \mc{U}$.  Thus $\mc{D} \cong \Spf W(k)[[T]]$.  

By Grothendieck's theory of tame lifting (\cite[XIII, Corollaire 2.12]{sga1}), the $G$-cover $f: V \to U$ 
lifts to a tame $G$-cover $\mc{F}: \mc{V} \to \mc{U}$ of formal schemes.  Applying \cite[Corollaire 2.14]{Ga:pr} shows that there exists
$R/W(k)$ finite such that $\mc{F}$, when base changed to $R$, extends to a $G$-cover $\mc{Y} \to \mc{X}$ of normal formal schemes,
where the special fiber of $\mc{Y}$ is possibly singular.
The equation of this $G$-cover over $\mc{D}$ is a $G$-extension $A_R/R[[T]]$ that is a birational lift of $A/k[[t]]$.
\end{proof}

\begin{remark}\label{Rconvoluted}
Of course, most of the real work is contained in \cite[Corollaire 2.14]{Ga:pr}.  For an overview of this corollary that uses less group theory than 
\cite{Ga:pr}, see \cite[Ch.\ 4]{Br:ac}, especially section 1.
The argument given above is somewhat artificial, as it is not actually necessary to globalize the problem by building the branched cover $f$.  
However, in \cite{Ga:pr}, the viewpoint is global, so we globalize in order to cite it.
\end{remark}

\begin{remark}\label{Rtowers}
Sa\"{i}di has refined the arguments of \cite{Ga:pr} to show that the birational local lifting problem is solvable in \emph{towers}.  More
specifically, let $B/k[[t]]$ be a $G$-Galois extension, let $\Gamma$ be a normal subgroup of $G$, and let $A = B^{\Gamma}$.  
If $A_R/R[[T]]$ is a birational lift of $A/k[[t]]$, then there is a finite extension $R'/R$ and a birational lift $B_{R'}/R'[[T]]$ of $B/k[[t]]$ 
such that $(B_{R'})^{\Gamma} \cong A_R \otimes_R R'$.  See
\cite[Theorem 2.5.5]{Sa:fl}, where birational lifts are called ``Garuti lifts." 
\end{remark}

The following criterion is extremely useful for seeing when a birational lift is actually a lift (i.e., when $A_R \otimes_R k$ is already integrally 
closed, thus isomorphic to $A$, in the language of Question \ref{Qbirational}).

\begin{prop}[The different criterion]\label{Pdifferent}
Suppose $A_R/R[[T]]$ is a birational lift of the $G$-Galois extension $A/k[[t]]$.  Let $K = \Frac(R)$, let $\delta_{\eta}$ be the degree of the different 
of $(A_R \otimes_R K)/(R[[T]] \otimes_R K)$, and let $\delta_s$ be the degree of the different of $A/k[[t]]$.  Then $\delta_s \leq \delta_{\eta}$, and 
equality holds if and only if $A_R/R[[T]]$ is a lift of $A/k[[t]]$. 
\end{prop}

\begin{proof}[Sketch of proof](cf.\ \cite[I, 3.4]{GM:lg})
For any free module $C/B$ of rank $r$, write $\det(C/B) = \bigwedge^r_B(C)$.
Consider the trace map 
$$\tau: \det(A_R/R[[T]]) \times \det(A_R/R[[T]]) \to R[[T]]$$ induced by
the map $A_R \times A_R \to R[[T]]$ given by $(x, y) \mapsto \trace_{A_R/R[[T]]}(xy).$  If $\pi$ is a uniformizer of $R$, then by the Weierstrass 
Preparation 
Theorem, the image of $\tau$ in $R[[T]]$ is generated by $\pi^nP(T)$, with $n \geq 0$ and $P(T)$ a distinguished polynomial 
(i.e., $P$ has a unit leading coefficient, and all other coefficients are in the maximal ideal of $R$).  Tensoring the map $\tau$ with $K$
shows that $\delta_{\eta} = \dim_{R[[T]] \otimes K}(\coker(\tau \otimes_R K)) = \deg P$.  But setting $\tau' = (\pi^{-n}\tau) \otimes_R k$, we see that the 
image of 
$$\tau': \det((A_R \otimes_R k)/k[[t]]) \times \det((A_R \otimes_R k)/k[[t]]) \to k[[t]]$$ is generated by the reduction of $P(T)$ to characterstic $p$, 
so $\dim_{k[[t]]}(\coker \tau') = \deg P = \delta_{\eta}$.

Let $\tilde{\tau}$ be the trace map $\det(A/k[[t]]) \times \det(A/{k[[t]]}) \to k[[t]]$.  Then $\dim_{k[[t]]}(\coker \tilde{\tau}) = \delta_s$.
Since $A$ is the integral closure of $A_R \otimes_R k$, \cite[III, Proposition 5]{Se:lf} shows that $\delta_s \leq \delta_{\eta}$, with equality
if and only if $A_R \otimes_R k \cong A$.
\end{proof}

\begin{remark}\label{Rfakelifting}
In \cite{Sa:fl}, Sa\"{i}di introduces the idea of a \emph{fake lifting} of a $G$-extension.  A fake lifting of a $G$-extension $A/k[[t]]$ is a birational lift that is
not an actual lift, but such that the degree of the different on the generic fiber is minimal among birational lifts.  
By Proposition \ref{Pdifferent}, the extension $A/k[[t]]$
lifts to characteristic zero iff it has no fake liftings.  Thus, one can try to answer a lifting problem in the affirmative by showing that the properties
that fake liftings must satisfy are so restrictive that they cannot possibly simultaneously hold.  
In \cite{Sa:fl}, this is applied to give a proof that $\ints/p$ is a local Oort group, and
also to show that certain $\ints/p^2$-extensions lift to characteristic zero (these results were already known---see \S\ref{Scyclic}).
\end{remark}

\section{Obstructions to lifting}\label{Sobstructions}
In \S\ref{Slifting}, we discussed some obstructions to the (global) lifting problem.  In this section, we will discuss the 
\emph{Katz-Gabber-Bertin obstruction}, or \emph{KGB obstruction} (\cite{CGH:ll}), which is the most effective way of showing that a particular local 
lifting problem does not have a solution.   

\subsection{The KGB obstruction}\label{SKGB}
Let $A/k[[t]]$ be a $G$-extension, where $G \cong P \rtimes \ints/m$, with $P$ a $p$-group and $p \nmid m$.  A theorem of Katz and Gabber
(\cite[Theorem 1.4.1]{Ka:lg}) states that there exists a unique $G$-cover $Y \to \proj^1_k$ that is \'{e}tale outside $t \in \{0, \infty\}$, tamely ramified of 
index $m$ above $t = \infty$, and
totally ramified above $t = 0$ such that the extension of complete local rings at $t=0$ is given by $A/k[[t]]$.  This is called the 
\emph{Katz-Gabber cover associated to $A/k[[t]]$}.  By Theorem \ref{Tlocalglobal} and
Proposition \ref{Pprimetop} (which does not depend on anything in this section), 
the $G$-cover $f: Y \to \proj^1_k$ lifts to characteristic zero iff the extension $A/k[[t]]$ does.  

Let $R/W(k)$ be finite, and let $K = \Frac(R)$.  Suppose $f_R: Y_R \to \proj^1_R$ is a lift of $f$ to characteristic zero.  
Since genus is constant in flat families, the genus of $Y$ is equal to that of $Y_K := Y_R \times_R K$.  
Furthermore, if $H \leq G$ is any subgroup, then $Y_R/H$ is a lift of $Y/H$ with generic fiber $Y_K/H$, and the genus of $Y/H$ is equal to that of
$Y_K/H$.  

\begin{definition}\label{DKGB}
Let $A/k[[t]]$ be a $G$-extension with associated Katz-Gabber $G$-cover $Y \to \proj^1_k$.  Then the \emph{KGB obstruction 
vanishes for $A/k[[t]]$} if there exists a $G$-cover $X \to \proj^1$ over a field of characteristic zero such that, for each subgroup $H \subseteq G$, the
genus of $Y/H$ is equal to the genus of $X/H$.  Equivalently (by the Hurwitz formula), the degree of the ramification divisor of $Y/H \to \proj^1_k$ is 
equal to that of $X/H \to \proj^1$ for each subgroup $H \subseteq G$.  
Equivalently, the degree of the ramification divisor of $Y \to Y/H$ is equal to that of $X \to X/H$ for each subgroup $H \subseteq G$.
\end{definition} 

Clearly, if $A/k[[t]]$ lifts to characteristic zero, its KGB obstruction must vanish.  If $G \cong P \rtimes \ints/m$ as above and
the KGB obstruction vanishes for \emph{all} $G$-extensions $A/k[[t]]$, then $G$ is called a \emph{KGB group} for $k$.

The following classification of all the KGB groups is due to Chinburg, Guralnick, and Harbater.

\begin{theorem}[\cite{CGH:ll}, Theorem 1.2]\label{TKGB}
The KGB groups for $k$ consist of the cyclic groups, the dihedral group $D_{p^n}$ for any $n$, the group $A_4$ (for $\cf(k) = 2$),
and the generalized quaternion groups $Q_{2^m}$ of order $2^m$ for $m \geq 4$ (for $\cf(k) = 2$).
\end{theorem}

\begin{remark}
There is another obstruction, called the \emph{Bertin obstruction} (\cite{Be:ol}), 
that was proven in \cite{CGH:ll} to be strictly weaker than the KGB obstruction.
Since it is also more difficult to describe, we will not discuss it here.
\end{remark}

Since any local Oort group must be a KGB group, one can ask which of the KGB groups are local Oort groups.  
The generalized quaternion groups were shown \emph{not} to be local Oort groups in \cite{BW:ac}.  The obstruction developed in \cite{BW:ac} 
to show this is called the \emph{Hurwitz tree obstruction}, and has to do with the non-existence of a kind of generalized Hurwitz tree
(for basics on Hurwitz trees, see \S\ref{Shurwitz}).

The group $A_4$ was announced to be a local Oort group in \cite{BW:ll}. 
The group $D_p$ was shown to be a local Oort group by Bouw and Wewers
(\cite{BW:ll}) for $p$ odd and by Pagot (\cite{Pa:ev}) for $p=2$.  We will discuss dihedral groups more in \S\ref{Smetacyclic}.  

For cyclic groups, a major guiding problem in the field is the \emph{Oort conjecture}.
\begin{conjecture}[Oort Conjecture]\label{Coort}
Any cyclic group is a local Oort group for any field $k$.
\end{conjecture}
If the cyclic group $G$ has $v_p(|G|) = 0$ (resp.\ $v_p(|G|) = 1$, resp.\ $v_p(|G|) = 2$), then $G$ is a local Oort group by Proposition 
\ref{Pprimetop} (resp.\ \cite{OSS:ask}, resp.\ \cite{GM:lg}).  These results, along with the Oort conjecture in general, will be discussed in detail in
\S\ref{Scyclic}.  We note that the Hurwitz tree obstruction to the Oort conjecture was shown to vanish in \cite[Ch.\ 4]{Br:rt}.

We also have the following strengthening of the Oort conjecture:

\begin{conjecture}[Strong Oort Conjecture]\label{Cstrongoort}
If $G$ is a cyclic group, then any $G$-extension $A/k[[t]]$ lifts over $W(k)[\zeta_{|G|}]$, where $\zeta_{|G|}$ is a primitive $|G|$th
root of unity.
\end{conjecture}

In fact, the results in \cite{OSS:ask} and \cite{GM:lg} prove the strong Oort conjecture for $v_p(|G|) \leq 2$.

\begin{remark}\label{Rdihedraloort}
In seems plausible that all $D_{p^n}$ might be local Oort groups for any $k$ of characteristic $p$, but no case
has been proven with $n > 1$ and no one has been willing to conjecture this.  See Question \ref{Qmetacyclic}.
\end{remark}

\begin{remark}\label{Roort}
The Oort conjecture was originally phrased in \cite[\S7]{Oo:ac} as the global statement that ``it seems reasonable to expect that [lifting] is
possible for every automorphism of an algebraic curve."
\end{remark}

\subsection{Non-KGB groups}\label{SnonKGB}
For a group $G$ that is not a KGB group, we can ask when a $G$-extension has vanishing KGB obstruction.  
In this section, we will focus on two of the simplest cases, namely $G = \ints/p \times \ints/p$, for
$p > 2$, and $G = \ints/p^n \rtimes \ints/m$, when $p \nmid m$.  We will use the fact that any $n$-tuple of generators
$(g_1, \ldots, g_n)$ of $G$ such that $g_1\cdots g_n = id$ gives rise to a $G$-cover of $\proj^1_K$ branched at $n$ points
for any algebraically closed field $K$ of characteristic zero (see, e.g., \cite{CH:hu}).  The branching indices of the $n$ branch points are the orders of 
$g_1, \ldots, g_n$, respectively (in fact, the inertia groups above these points are conjugates of the cyclic 
groups generated by the $g_i$).  The $n$-tuple $(g_1, \ldots, g_n)$ is called a \emph{branch cycle description} of the $G$-cover.

\begin{prop}\label{Pzpzp}
 Let $m_1$ and $m_2$ be the first and second lower jumps of a $\mathbb{Z}/p
\times \mathbb{Z}/p$-extension $A/k[[t]]$. If the KGB obstruction for
$A/k[[t]]$ vanishes, then $m_1 \equiv -1 \pmod{p}$.  The converse is
true unless $p = 3$ and $m_1 = m_2 = 2$.
\end{prop}

\begin{proof} (cf.\ \cite[I, Theorem 5.1]{GM:lg})
The degree of the ramification 
divisor of the Katz-Gabber cover $g: Y \to \proj^1_k$ associated to $A/k[[t]]$ is the same as the degree of the different of $A/k[[t]]$, which is 
$\delta = (m_1 + 1)(p^2-1) + (m_2 - m_1)(p-1)$ by (\ref{Ebasicdifferent}).  Now, each branch point of a $\ints/p \times \ints/p$-cover in characteristic 
zero is 
branched of index $p$, and thus contributes $p(p-1)$ to the degree of the ramification divisor.  Thus, for the KGB obstruction to vanish, we must
have $p(p-1) \, | \, \delta$.  This is equivalent to $p \, | \, m_2 + 1$.  Since $m_1 \equiv m_2 \pmod{p}$ (\cite[IV, Prop.\ 11]{Se:lf}), we have
$m_1 \equiv -1 \pmod{p}$.

Conversely, suppose $m_1 \equiv -1 \pmod{p}$ (and thus $m_2 \equiv -1 \pmod{p}$ as well).
Consider a branch cycle description of the form
$$(g_{0,1}, \ldots, g_{0, r_0}, g_{1,1}, \ldots, g_{1, r_2}, \ldots, g_{p, 1}, \ldots, g_{p, r_{p}})$$ 
corresponding to branch points
$$(x_{0,1}, \ldots, x_{0, r_0}, x_{1,1}, \ldots, x_{1, r_2}, \ldots, x_{p, 1}, \ldots, x_{p, r_{p}}),$$ 
where $r_0 = \cdots = r_{p-1} = \frac{m_1+1}{p}$ and $r_{p} = \frac{m_2+1}{p}$. 
Write $G := \ints/p \times \ints/p$ additively, such that $(1, 0)$ is an element of $G_{m_2}$. 
Then we can take each $g_{i,j}$ to be a non-identity multiple of $(i, 1)$ for $0 \leq i \leq p-1$, and each $g_{p, j}$ to be a 
non-identity multiple of $(1, 0)$.  
We leave it as an exercise to show that, unless $p=3$ and $m_1 = m_2 = 2$, these
choices can be made so that $\sum_{i,j} g_{i,j} = (0,0)$.
This gives rise to a $G$-cover $X \to \proj^1_K$, with $\cf(K) = 0$.

Suppose $H$ is a subgroup of $G$.  We calculate the degrees of the ramification divisors $\mc{R}_X$ of $X \to X/H$ and 
$\mc{R}_Y$ of $Y \to Y/H$.  If $H$ is trivial, then these degrees are both zero.
If $H = G$, then $|\mc{R}_Y| = (m_1 + 1)(p^2-1) + (m_2 - m_1)(p-1)$ by the same calculation as earlier 
in the proof, and it is easy to see that $|\mc{R}_X| = (m_1 + 1 + \frac{m_2 +1}{p})(p^2-p)$, which is equal to $|\mc{R}_Y|$.  

Now assume $H$ has order $p$. The $H$-cover $X \to X/H$ is ramified exactly at the ramification points above those $x_{i,j}$ for which $g_{i, j}$
generates $H$.  Since $p$ points of $X$ lie above each $x_{i,j}$, we obtain $|\mc{R}_X| = (m_2+1)(p-1)$ if $H$ is generated by $(1,0)$, and 
$(m_1+1)(p-1)$ if not.

If the $H_j$ are the higher ramification subgroups for the lower numbering for $Y \to Y/H$ at the ramified point, 
and the $G_j$ are the higher ramification groups for $Y \to \proj^1_k$, then $H_j = H \cap G_j$ (\cite[Prop.\ 2]{Se:lf}).  
By (\ref{Ebasicdifferent}), we have $|\mc{R}_Y| = \sum_{j\geq 0} |H_j - 1|$.  Note that $H \cap G_j \cong \ints/p$ for $j \leq m_1$, and
is trivial when $j > m_1$, except if $H$ is generated by $(1,0)$, in which case $H_j \cong \ints/p$ for $m_1 < j \leq m_2$.
So one sees that $|\mc{R}_Y| = (m_2 + 1)(p-1)$ if $H$ is generated by $(1, 0)$, or $(m_1 + 1)(p-1)$ if not.  Thus, the $G$-cover $X \to \proj^1$ causes
the KGB obstruction to vanish.
\end{proof}

\begin{prop}\label{Pzpzm}
Let $G = \ints/p^n \rtimes \ints/m$, where $G$ is \emph{not} cyclic and $p \nmid m$.  
Consider a $G$-extension $A/k[[t]]$ with first positive lower jump $h$.
Then the KGB obstruction for $A/k[[t]]$ vanishes iff $h \equiv -1 \pmod{m}$.  Furthermore, for this to happen, we must have that
the conjugation action of $\ints/m$ on $\ints/p^n$ is faithful.
\end{prop}

\begin{proof}
(cf. \cite[Prop.\ 1.3]{BW:ll}) 
Let $g: Y \to \proj^1_k$ be the Katz-Gabber cover associated to $A/k[[t]]$.  Let $f: X \to \proj^1_K$ be a $G$-cover in characteristic zero.
Now, $Y/(\ints/p^n) \to \proj^1_k$ is a $\ints/m$-cover, branched at two points, so $Y/(\ints/p^n)$ must have genus $0$.  
Thus, if $f$ is to be a witness to the 
vanishing of the KGB obstruction, then the $\ints/m$-cover $X/(\ints/p^n) \to \proj^1_K$ must have genus $0$.  Since any element of  
$PGL_2(K)$ of finite order is conjugate (after a possible extension of $K$) to a diagonal element when $\cf(K) = 0$, 
such an element acts on $\proj^1$ with exactly two fixed points.
Thus $X/(\ints/p^n) \to \proj^1_K$ must be branched at two points with order $m$.  
Since $G$ is not cyclic, any element of $G$ with order divisible by $m$ in fact has order $m$.  So the $G$-cover $f$ 
must have two branch points
of index $m$, along with some branch points of $p$-power index.  Consider $\tilde{f}: X/(\ints/p^{n-1}) \to \proj^1_K$.  
The degree of the ramification divisor
of this branched cover is $2(mp - p) + r(mp - m)$, where $r$ is the number of branch points of $f$ with index $p^n$ (these points are branched of index
$p$ in $\tilde{f}$).  

Now, the first positive upper jump for $A/k[[t]]$ is $h/m$, and since the upper numbering is preserved under quotients (\cite[Prop.\ 14]{Se:lf}), 
the first positive upper
jump for $A^{\ints/p^{n-1}}/k[[t]]$ is $h/m$.  So the first lower jump for $A^{\ints/p^{n-1}}/k[[t]]$ is $h$.  The degree of the different of this extension
is then $mp-1 + h(p-1)$, by (\ref{Ebasicdifferent}).  Thus the degree of the ramification divisor of $Y/(\ints/p^{n-1}) \to \proj^1_k$ is 
$(mp - p) + mp - 1 + h(p-1)$.  Equating this to $2(mp - p) + r(mp - m)$ yields $h = -1 + rm$, so $h \equiv -1 \pmod{m}$.  
Furthermore, it is well-known (see, e.g., \cite[Lemma 1.4.1(iv)]{Pr:fc}) that $(h, m) = 1$ implies that the action of $\ints/m$ on the $p$-Sylow
subgroup of $G/(\ints/p^{n-1})$ is faithful.  Thus the action of $\ints/m$ on $\ints/p^n \subseteq G$ is faithful.

Conversely, suppose $h \equiv -1 \pmod{m}$ (so the action of $\ints/m$ on $\ints/p^n$ is faithful).  
Let the positive upper jumps for $A/k[[t]]$ be $u_1 = h/m, u_2, \ldots, u_n$.  
By \cite[Theorem 1.1]{OP:wc}, we have that $mu_1 \equiv \cdots \equiv mu_n \equiv -1 \pmod{m}$.
Consider a branch cycle description of the form
$$(\gamma_1, \gamma_2, g_{1,1}, \ldots, g_{1, r_1}, g_{2,1}, \ldots, g_{2, r_2}, 
\ldots, g_{n, 1}, \ldots, g_{n, r_n})$$ corresponding to branch points
$$(z_1, z_2, x_{1,1}, \ldots, x_{1, r_1}, x_{2,1}, \ldots, x_{2, r_2}, 
\ldots, x_{n, 1}, \ldots, x_{n, r_n}),$$
where $r_1 = u_1 + \frac{1}{m}$ and $r_i = u_i - u_{i-1}$ for $2 \leq i \leq n$.
We choose $\gamma_1$ and $\gamma_2$ to be elements of $G$ of order $m$, and for $1 \leq i \leq n$, we choose each $g_{i, j}$ to be an element of
$G$ of order $p^{n-i+1}$.  By choosing $\gamma_1$ correctly, we can ensure that $\gamma_1\gamma_2\prod_{i,j}g_{i,j}$ is the identity.  
This gives rise to a $G$-cover $X \to \proj^1_K$, with $\cf(K) = 0$.

Suppose $H$ is a subgroup of $G$ of order $p^{n'}m'$.  We calculate the degrees of the ramification divisors $\mc{R}_X$ of $X \to X/H$ and 
$\mc{R}_Y$ of $Y \to Y/H$.
We get a contribution of $p^{\min(n', n-i+1)} - 1$ to $|\mc{R}_X|$ at each of the $mp^{i-1}$ points of $X$ lying above an $x_{i, j}$, for any $j$.  
We get a contribution of $m' - 1$ to $|\mc{R}_X|$ at each point above $z_1$ or $z_2$ whose inertia group intersects $H$ nontrivially.
The $p^n$ subgroups of $G$ of order $m$ are all conjugate and disjoint except for the identity element, and each one is the inertia group
at one point above $z_1$ and one point above $z_2$.  Furthermore, $H$ contains $p^{n'}$ subgroups
of order $m'$, each one contained in a unique subgroup of order $m$ of $G$.  So the total contribution to $|\mc{R}_X|$ for the points above
$z_1$ and $z_2$ is $2p^{n'}(m'-1)$.  We have shown:
$$|\mc{R}_X| = 2p^{n'}(m'-1) + \sum_{i=1}^n r_i mp^{i-1} (p^{\min(n', n-i+1)} - 1).$$ 

In characteristic $p$, 
the tamely ramified point of $Y \to \proj^1_k$ contributes $p^{n'}(m'-1)$ to $|\mc{R}_Y|$ for the same reasons as in the last paragraph.
Using Herbrand's formula, the higher ramification filtration of $Y \to \proj^1_k$ at the wildly ramified point for the lower numbering 
has one group of order
$p^nm$, and $mp^{i-1}(u_i - u_{i-1})$ groups of order $p^{n-i+1}$ for $1 \leq i \leq n$, where we set $u_0 = 0$.  If the $H_j$ are the 
higher ramification subgroups for $Y \to Y/H$ at this point, and the $G_j$ are the higher ramification groups for
$Y \to \proj^1_k$, then $H_j = H \cap G_j$ (\cite[Prop.\ 2]{Se:lf}).  By (\ref{Ebasicdifferent}), the contribution to $|\mc{R}_Y|$ at this point
is $\sum_{j\geq 0} |H_j - 1|$.  Note that if $j > 0$, then $|H_j|$ is the smaller of the order of the $p$-Sylow subgroups of $H$ and $G_j$.
We conclude that
$$|\mc{R}_Y| = p^{n'}(m'-1) + p^{n'}m' - 1 + \sum_{i=1}^n mp^{i-1}(u_i - u_{i-1})(p^{\min(n', n-i+1)} - 1).$$ 
Some rearrangement shows that $|\mc{R}_X| = |\mc{R}_Y|$.  Thus the KGB obstruction vanishes.
\end{proof}

\begin{remark}\label{Rdihedral}
If $p$ is odd, then any action of the dihedral group $D_{p^n}$ satisfies the condition of Proposition \ref{Pzpzm} 
(\cite[Lemma 1.4.1]{Pr:fc} or \cite[Theorem 1.1]{OP:wc}).  So Proposition \ref{Pzpzm} shows that $D_{p^n}$ is a KGB group for $p$ odd.
\end{remark}

\begin{question}\label{Qmetacyclic}
Does every $\ints/p^n \rtimes \ints/m$-action satisfying the condition of Proposition \ref{Pzpzm} lift to characteristic zero?   
\end{question}

Question \ref{Qmetacyclic} has a positive answer for $n=1$ (Theorem \ref{Tzpzm}).  For $n>1$, however, I am not aware of a single 
$G$-extension in this form that is known either to lift or not to lift.
But there are $D_4$-extensions that are known to lift from characteristic $2$ to characteristic zero (\cite{Br:D4}).

\subsection{Weak Oort groups}\label{Sweak}
A group $G$ is called a \emph{weak Oort group} for $k$ if there exists a $G$-extension $A/k[[t]]$ that lifts to characteristic zero.  For instance,
we mentioned above that $D_4$ is a weak Oort group for any algebraically closed field $k$ of characteristic $2$.  Clearly, any local Oort
group is a weak Oort group.  We will content ourselves here to state two results about weak Oort groups, one positive and one negative.
The first is due to Matignon:

\begin{prop}[\cite{Ma:pg}]\label{Pzpnweak}
The group $(\ints/p)^n$ is a weak Oort group for all primes $p$ and positive integers $n$.
\end{prop}

Note that, if $n > 1$ and $p$ is odd, then we can build a $(\ints/p)^n$-extension with a $(\ints/p)^2$-subextension that has nonvanishing KGB 
obstruction (Proposition \ref{Pzpzp}).  Thus $(\ints/p)^n$ is not an Oort group (or even a KGB group) by Theorem \ref{TKGB} 
(one also sees it is not an Oort group from the example in \S\ref{Slifting}).  For $p = 2$, we have that $(\ints/p)^n$ is not an Oort group for $n > 2$ by
the example in \S\ref{Slifting}. 
 
The second result follows from \cite[Theorem 1.8]{CGH:ll}.
\begin{prop}
If $G$ contains an abelian subgroup that is neither cyclic nor a $p$-group, then $G$ is not a weak Oort group. 
\end{prop}

\section{Cyclic groups}\label{Scyclic}
Throughout \S\ref{Scyclic}, we take $R$ to be a large enough finite extension of $W(k)$, and $K = \Frac(R)$. 
Let $\mc{D} = \Spec R[[T]]$ be the open unit disc.  If $L$ is any field of characteristic $p$, we say that $x \in L$ is a \emph{$\wp$th power}
if the equation $y^p - y = x$ has a solution in $L$.  In this case $y$ is called a \emph{$\wp$th root} of $x$.

The purpose of \S\ref{Scyclic} is to discuss progress toward the Oort conjecture (Conjecture \ref{Coort}).

\subsection{Reduction to case of cyclic $p$-groups}\label{Spgroup}
Cyclic groups of prime-to-$p$ order are local Oort groups:

\begin{prop}\label{Pprimetop}
Suppose $x$ is any $K$-point of $\mc{D}$.  
If $A/k[[t]]$ is a $\ints/m$-extension, with $p \nmid m$, then it lifts to a $\ints/m$-extension $A_R/R[[T]]$.  We can choose this lift so that if 
$f: \Spec A_R \to \mc{D}$ is the induced map on schemes, then the generic fiber of $f$ is branched only at $x$.
If no branching behavior is specified, we can take $R = W(k)$.
\end{prop}

\begin{proof}
By Kummer theory, we may assume without loss of generality that $A$ is given by $k[[t]][y]/(y^m - t)$.  Think of $x$ as an element of $R$ with 
$v(x) > 0$ (see \S\ref{Sdiscs}).  Then choose $A_R$ to be given by $R[[T]][Y]/(Y^m - (T - x))$.  The reduction of $A_R$ is as desired.
If $x$ is not specified, then the lift $R[[T]][Y]/(Y^m - T)$ is defined over $W(k)$.
\end{proof}  

\begin{remark}
Proposition \ref{Pprimetop}, along with the local-global principle (Theorem \ref{Tlocalglobal}), gives another proof of Proposition 
\ref{Ptame}.
\end{remark}

Using Proposition \ref{Pprimetop}, we can reduce the local lifting problem for cyclic groups to the $p$-group case:
\begin{prop}\label{Preducetop}
Let $G \cong \ints/mp^n$, where $p \nmid m$.  If $\ints/p^n$ is a local Oort group, then $G$ is a local Oort group.
\end{prop}

\begin{proof}
Let $H$ be the unique subgroup of $G$ of order $m$ and let $H'$ be the unique subgroup of order $p^n$.   
Given a $G$-cover $f: Y \to \Spec k[[t]]$, let $g: X \to \Spec k[[t]]$ (resp.\ $g': X' \to \Spec k[[t]]$) be the unique quotient $\ints/m$-extension 
(resp.\ $\ints/p^n$-extension).
Then the normalization of $X \times_{\Spec k[[t]]} X'$ is isomorphic to $Y$.  By assumption,
$g'$ lifts to a $\ints/p^n$-cover $g_R': X'_R \to \mc{D}$, with generic fiber $g_K'$.  
By Proposition \ref{Pprimetop}, $g$ lifts to a $\ints/m$-cover $g_R: X_R \to \mc{D}$, and
we can choose $g$ so that the unique branch point of the generic fiber $g_K$ of $g_R$ is branched of index $p^n$ 
in the generic fiber $g'_K$ of $g'$.

Let $X_R''$ be the normalization of $X_R \times_{\mc{D}} X_R'$.  Then the canonical map $f_R: X_R'' \to \mc{D}$ is clearly a birational
lift of $f$.  The degree of the different of $g$ (and of $g_K$) is $m-1$.  
Let $\delta$ be the degree of the different of $g'$ (and of $g_K'$).  
Using our assumptions on the branch loci of $g_K$ and
$g_K'$, one calculates the degree of the different of the generic fiber $f_K$ of $f_R$ to be
$m\delta + m - 1$.  On the other hand, the higher ramification groups $(H')^i$ for the upper numbering at the unique ramification point of
$g$ are the same as the corresponding groups $G^i$ for $f$, as long as $i > 0$ (\cite[IV, Prop.\ 14]{Se:lf}).  Also, $|(H')^0| = p^n$
and $G^0 = mp^n$.  Applying Herbrand's formula shows that $|H'_0| = p^n$, $|G_0| = mp^n$, and $|H'_i| = G_{mi}$ for $i \geq 1$.  Then
(\ref{Ebasicdifferent}) shows that the degree of the different of $f$ is also $m\delta + (m-1)$.  By Proposition 
\ref{Pdifferent}, we conclude that $f_R$ is a lift of $f$.
\end{proof}

\begin{remark}
For any $i \in \nats$, let $\zeta_i$ be a primitive $i$th root of unity.
One would like to claim that if the strong Oort conjecture (Conjecture \ref{Cstrongoort}) holds for $\ints/p^n$, then it also holds for $\ints/mp^n$, 
for any $m$ with $p \nmid m$.  However, if the lift of $g'$ in Proposition \ref{Preducetop} can be done over $W(k)[\zeta_{p^n}]$,
then the lift of $f$ in Proposition \ref{Preducetop} can be done over $W(k)[\zeta_{p^n}, \zeta_m] = W(k)[\zeta_{mp^n}]$ only if $g'_K$ has
a branch point of index $p^n$ defined over $\Frac(W(k)[\zeta_{mp^n}])$.  It does not seem clear that this must be the case.
\end{remark}

\subsection{$\ints/p^n$-extensions in characteristic $p$}\label{Scharp}
In attempting to prove the Oort conjecture, it is useful to have a somewhat explicit form for all $\ints/p^n$-extensions of $k((t))$.
It is well-known, by Artin-Schreier theory, that any $\ints/p$-extension of $k((t))$ is given by an equation $y^p - y = f(t)$,
where $f(t)$ is not a $\wp$th power in $k((t))$.  It is clear that $f(t)$ must have some term with negative degree, as otherwise, letting 
$\tilde{f} = f - f(0)$, we have that  
$a - \tilde{f} - \tilde{f}^p - \tilde{f}^{p^2} - \cdots$ is a $\wp$th root of $f$, where $a$ is a $\wp$th root of $f(0)$.  
In fact, after a change of variable in $t$, we may assume $f(t) = t^{-j}$, and in 
this case $j$ is the unique upper jump in the higher ramification filtration of the extension (see Proposition \ref{Ppnexplicit} below).  

It is less well-known that there is an explicit form for all $\ints/p^n$-extensions of $k((t))$, written as successive Artin-Schreier 
extensions.  The higher ramification filtration
(and thus the degree of the different) can also be read off from this form.  

\begin{prop}\label{Ppnexplicit}
There exist explicit polynomials $f_1, \ldots, f_{n-1}$ over $\FF_p$ in $1, \ldots, n-1$ variables respectively, such that for any given 
$\ints/p^n$-extension $L/k((t))$, there is a choice of $x_1, \ldots, x_n \in k((t))$ such that $L/k((t))$ is given (possibly after a change of variable in $t$)
by adjoining $y_1,\ldots, y_n$ satisfying
\begin{eqnarray*}
y_1^p - y_1 &=& x_1 \\
y_2^p - y_2 &=& f_1(y_1) + x_2 \\
&\vdots& \\
y_n^p - y_n &=& f_{n-1}(y_1, \ldots, y_{n-1}) + x_n.
\end{eqnarray*}
Furthermore, we can choose $x_1 = t^{-j}$, for some $j \geq 1$ not divisible by $p$, and we can choose $x_2, \ldots, x_n$ to be polynomials 
in $t^{-1}$ with no terms of degree divisible by $p$.  Conversely, every choice of $x_1, \ldots, x_n$ in this form gives a distinct 
$\ints/p^n$-extension.  
If $(u_1, \ldots, u_n)$ are the upper jumps of $L/k((t))$, then $u_1 = j$, and for $i > 1$, we have
$u_i = \max(\deg(x_i), pu_{i-1})$.

We can pick a generator $\sigma$ of $\ints/p^n$ such that $\sigma^{p^{i-1}}y_i = y_i + 1$ for $1 \leq i \leq n$. 
\end{prop}

\begin{proof}
That the equations and action of $\sigma$ can be chosen in this form is originally due to Schmid (\cite{Sc:af}), see also \cite[\S3]{OP:wc}.  
The formula for the upper jumps was proven by Garuti (\cite[Theorem 1.1]{Ga:ls}).
\end{proof}

Proposition \ref{Ppnexplicit} has the immediate corollary:
\begin{corollary}\label{Ccomparejumps}
The upper jumps $(u_1, \ldots, u_n)$ of the extension above satisfy $u_i \geq pu_{i-1}$ for $2 \leq i \leq n$.  If $u_i > pu_{i-1}$, then 
$p \nmid u_i$.
\end{corollary}

\begin{remark}\label{Rwitt}
The \emph{Artin-Schreier-Witt theory} (detailed in \cite[\S3]{OP:wc} and \cite[pp.\ 330-331]{La:al})
says that for any field $L$ of characteristic $p$, there is an association between truncated Witt vectors 
$x \in W_n(L)$ of length $n$ and $\ints/p^n$-extensions of $L$. 
Furthermore, $x, y \in W_n(L)$ are associated to the same extension iff there exists $z \in W_n(L)$ with $x - y = F(z) - z$, 
where $F:W_n(L) \to W_n(L)$ is the Frobenius map.  
If $L = k((t))$, then the length $n$ Witt vector $(x_1, \ldots, x_n) \in W_n(L)$ is associated to the extension in 
\ref{Ppnexplicit}.
\end{remark}

\subsection{Lifting $\ints/p$-extensions}\label{Szplifting}
The following proposition solves the local lifting problem for $\ints/p$-extensions.

\begin{theorem}\label{Tzplifting}
The group $\ints/p$ is a local Oort group.  In particular, any $G$-extension $A/k[[t]]$ lifts over $W(k)[\zeta_p],$ where $\zeta_p$
is a primitive $p$th root of unity. 
\end{theorem}

\begin{proof}
By Proposition \ref{Ppnexplicit}, $A$ is the normalization of $k[[t]]$ in $L/k((t))$, where
$L/k((t))$ is given by $y^p - y = t^{-u_1}$, and $u_1$ is the upper jump.  Let $R= W(k)[\zeta_p]$, and let $\lambda = \zeta_p - 1$.
Then $v(\lambda^{p-1} + p) > 1$.  Consider the integral closure $A_R$ of $R[[T]]$ in the Kummer extension of $\Frac(R[[T]])$ 
given by $$Z^p = 1 + \lambda^p T^{-u_1}.$$  Making the substitution $Z = 1 + \lambda Y$, we obtain
$$(\lambda Y)^p + p\lambda Y  + o(p^{p/(p-1)})= \lambda^p T^{-u_1},$$ where $o(p^{p/(p-1)})$ represents terms with coefficients of valuation
greater than $\frac{p}{p-1}$.  This reduces to $y^p - y = t^{-u_1}$.  So we have constructed a birational lift.  

The degree of the different of $A/k[[t]]$ is $(u_1+1)(p-1)$, by Lemma \ref{Ljumpsdifferent}.
On the other hand, the generic fiber of $\Spec A_R \to \Spec R[[T]]$ is branched at exactly $u_1 + 1$ points in the 
unit disc ($T = 0$ and $T$ equals each $u_1$th root of $-\lambda^p$).  Since the ramification is tame, the degree of the different of $A/R[[T]]$ is 
$(u_1 + 1)(p-1)$ as well.  By Proposition \ref{Pdifferent}, our birational lift is an actual lift.
\end{proof}

\begin{remark}\label{Rzplifting}
Theorem \ref{Tzplifting}, along with Proposition \ref{Preducetop} and the local-global principle, shows that, for $G \cong \ints/pm$ with $p \nmid m$, 
all $G$-actions on curves lift to characteristic zero.  The original proof of this is due to Oort-Sekiguchi-Suwa (\cite{OSS:ask}), and is
global in nature.  In particular, the proof relies on intricate calculations involving extensions of group schemes and deformations of generalized 
Jacobians.  The local proof above, due to Green and Matignon (\cite{GM:lg}), is much simpler, once one admits the local-global principle.
\end{remark}

\subsection{Sekiguchi-Suwa Theory}\label{Ssstheory}
Underlying the calculation in the proof of Theorem \ref{Tzplifting} is the so-called ``Kummer-Artin-Schreier theory" in degree $p$.  Its generalization to 
degree $p^n$ by Sekiguchi and Suwa (\cite{SS:kasw1}, \cite{SS:kasw2}) is called the ``Kummer-Artin-Schreier-Witt theory."  While a full accounting of 
the theory is well beyond the scope of this paper, we will give a brief exposition below.  For a more detailed exposition (although still not as detailed as 
the papers of Sekiguchi and Suwa), see \cite{MRT:ss}.\\

\subsubsection{Kummer-Artin-Schreier theory.}\label{Skas}
Phrased in the language of group schemes, Kummer theory (of degree $p$) studies torsors under the group scheme $\mu_p$, whereas Artin-Schreier
theory studies torsors under the group scheme $\ints/p$ in characteristic $p$.  Now, if $A/k[[t]]$ is a $\ints/p$-extension with a lift 
$A_R/R[[T]]$, then $A_R/R[[T]]$ is a torsor under a finite, flat group scheme over
 $R$ whose generic fiber is $\mu_p$ and whose special fiber is $\ints/p$ (of course, if $R$ contains the $p$th roots of unity, then $\mu_p \cong \ints/p$,
 but it is more natural to think of the generic fiber as $\mu_p$).
 These are rank $p$ subgroup schemes of $\Gp_m$ and 
 $\Gp_a$, respectively.  So in order to understand lifts of $\ints/p$-extensions, one approach is to study group schemes
 over $R$ whose generic fiber is $\Gp_m$ and whose special fiber is $\Gp_a$.
 
For any $\mu \in \mf{m},$ the maximal ideal of $R$, consider the group scheme 
$$\mc{G}^{(\mu)} := \Spec(R[x, \frac{1}{1 + \mu x}])$$ with comultiplication law
$$(x, y) \mapsto \mu xy + x + y,$$ coinverse $$x \mapsto -\frac{x}{1 + \mu x},$$ and counit $x \mapsto 0.$
There is a map from $\mc{G}^{(\mu)} \to \Gp_m := R[x, \frac{1}{x}]$ given by $x \to 1+ \mu x$, which is an isomorphism on the generic fiber.  
Indeed, a good way to think about the $R$-points
of $\mc{G}^{(\mu)}$ is as elements of $R$ with multiplication defined by $x * y = z$, where $z$ is such that $1 + \mu z = (1 + \mu x) (1 + \mu y)$.
Clearly, the special fiber of $\mc{G}^{(\mu)}$ is isomorphic to $\Gp_a$.  

Suppose $R$ contains the $p$th roots of unity.  Let $\mu = \lambda$ as in the proof of Theorem \ref{Tzplifting}, so $\mu =  \zeta_p - 1$, 
where $\zeta_p$ is a nontrivial $p$th root of unity.  Then there is a group scheme 
morphism $\psi: \mc{G}^{(\lambda)} \to \mc{G}^{(\lambda^p)}$ given by
\begin{equation}\label{Epsi}
x \mapsto \frac{(1 + \lambda x)^p - 1}{\lambda^p},
\end{equation}
and this map is surjective in the \'{e}tale topology. 
Note that $\psi$ is the map that makes the following diagram commute, where the
vertical arrows are the maps to $\Gp_m$ in the previous paragraph:
\[
\xymatrix{
\mc{G}^{(\lambda)} \ar[r]^{\psi} \ar[d] & \mc{G}^{(\lambda^p)} \ar[d] \\
\Gp_m \ar[r]^{(\cdot)^p} & \Gp_m.
}
\]
So the generic fiber of $\psi$ is isomorphic to the $p$th power map on $\Gp_{m, K}$ and the special fiber of $\psi$, after a similar calculation
to the proof of Theorem \ref{Tzplifting}, is the map $x \mapsto x^p - x$ on $\Gp_{a, k}$.
Since $R$ contains the $p$th roots of unity, the kernel of $\psi$ is the constant group scheme $\ints/p$.  
Thus we have an exact sequence of \'{e}tale sheaves over $R$:
\begin{equation}\label{Ekas}
0 \to \ints/p \to \mc{G}^{(\lambda)} \to \mc{G}^{(\lambda^p)} \to 0.
\end{equation}
The generic fiber of (\ref{Ekas}) is just the Kummer exact sequence over $K$, whereas the special fiber is just the Artin-Schreier sequence over $k$.

\begin{prop}\label{Psszp}
Suppose $R$ contains the $p$th roots of unity.  If $B$ is a flat local $R$-algebra, then any unramified $\ints/p$-cover $\Spec C \to \Spec B$ is given by 
a Cartesian diagram of the form
\[
\xymatrix{
\Spec C \ar[r] \ar[d] & \mc{G}^{(\lambda)} \ar[d]^{\psi} \\
\Spec B \ar[r] & \mc{G}^{(\lambda^p)}
}
\]
That is, $C \cong B[Y]/(\frac{(1 + \lambda Y)^p - 1}{\lambda^p} - X))$, with $X \in B$.
\end{prop}

\begin{proof}
The short exact sequence (\ref{Ekas}) gives an exact sequence
$$0 \to \mc{G}^{(\lambda^p)}(B)/\mc{G}^{(\lambda)}(B) \to H^1(\Spec B, \ints/p) \to H^1(\Spec B, \mc{G}^{(\lambda)}).$$  
But $H^1(\Spec B, \mc{G}^{(\lambda)}) = 0$ (\cite[Theorem 2.2]{SS:kasw1}), so $\ints/p$-torsors over $\Spec B$ are in one-to-one
correspondence with elements $\mc{G}^{(\lambda^p)}(B)/\mc{G}^{(\lambda)}(B)$.  Such elements can be interpreted exactly as elements 
$X \in B$, where $X_1$ is considered equivalent to $X_2$ if $(1 + \lambda^p X_1)/(1 + \lambda^p X_2) = (1 + \lambda Y)^p$ for some $Y$ in $B$.
The proposition follows by the standard expression of the coboundary map from degree $0$ to degree $1$ in \'{e}tale cohomology.
\end{proof}

\begin{remark}\label{Rzpliftinginterp}
If $\pi$ is a uniformizer of $R$, and $B = R[[T]]_{(\pi)}$, then $B$ is a local $R$-algebra with residue field $k((t))$.
The lift exhibited in Theorem \ref{Tzplifting} comes (birationally) from choosing $X = T^{-u_1}$ in Proposition \ref{Psszp}.  So in some sense, we
can say that our lift of a $\ints/p$-extension comes from the Kummer-Artin-Schreier theory.  In \S\ref{Szp2lifting}, we will discuss how lifts of
$\ints/p^2$-extensions come from the Kummer-Artin-Schrier-Witt theory described below.
\end{remark}

\subsubsection{Kummer-Artin-Schreier-Witt theory.}\label{Skasw}
The point of Kummer-Artin-Schreier-Witt theory (or \emph{Sekiguchi-Suwa theory}) 
is to generalize the exact sequence (\ref{Ekas}) to a sequence of group schemes with kernel 
$\ints/p^n$, and to generalize Proposition \ref{Psszp} to $\ints/p^n$-covers.  In this section, we will concentrate on the theory itself, and in 
\S\ref{Szp2lifting}, we will show how this is applied to the lifting problem.  

Assume that $R$ contains the $p^n$th roots of unity.  
We seek an exact sequence of group schemes over $R$ whose generic fiber is the Kummer-like sequence 
\begin{equation}\label{Egenkummer}
1 \to \mu_{p^n} \cong \ints/p^n \to (\Gp_m)^n \stackrel{\phi_n}{\to} (\Gp_m)^n \to 1,
\end{equation} 
where the surjection $\phi_n$ is given by
$$(x_1, \ldots, x_n) \mapsto (x_1^p, x_2^p/x_1, x_3^p/x_2, \ldots, x_n^p/x_{n-1}),$$ 
and whose special fiber is the \emph{Artin-Schreier-Witt} sequence
\begin{equation}\label{Eartinschreierwitt}
1 \to \ints/p^n \to W_n \stackrel{\wp}{\to} W_n.
\end{equation}
Here $W_n$ is the scheme of length $n$ Witt vectors over $k$, which is an $n$-fold extension
of $\Gp_a$'s, and $\wp$ is the map $F - Id$, where $F$ is the Frobenius map.  In the Kummer-like sequence, since $R$ contains the $p^n$th 
roots of unity, we have $\mu_{p^n} \cong \ints/p^n$, and we can think of the injection as sending $1$ to $(\zeta_p, \zeta_{p^2}, \ldots, \zeta_{p^n})$,
where each $\zeta_{p^i}$ is a $p^i$th root of unity and $\zeta_{p^j}^{p^{j-i}} = \zeta_{p^i}$ for $j \geq i$.  

\begin{prop}[\cite{SS:kasw1}, Theorem 7.1 or \cite{SS:kasw2}, Theorem 8.1]\label{Pkaswtheory}
For each positive integer $n$, there exists a flat group scheme $\mc{W}_n$ over $R$ (called a \emph{Kummer-Artin-Schreier-Witt group scheme})
that fits into the exact sequence 
\begin{equation}\label{Ekasw}
0 \to \ints/p^n \to \mc{W}_n \stackrel{\psi_n}{\to} \mc{V}_n := \mc{W}_n/(\ints/p^n) \to 0
\end{equation}
of group schemes over $R$.
The special fiber of (\ref{Ekasw}) is isomorphic to the Artin-Schreier-Witt exact sequence (\ref{Eartinschreierwitt}) over $k$, and the generic fiber
is isomorphic to the Kummer-like exact sequence (\ref{Egenkummer}) over $K$.  Furthermore, $\mc{W}_n$ is an $n$-fold extension of 
$\mc{G}^{(\lambda)}$'s, and as a scheme is given by
$$\Spec R[Y_1, \ldots, Y_n, \frac{1}{1+\lambda Y_1}, \frac{1}{F_1(Y_1)+\lambda Y_2}, \ldots, \frac{1}{F_{n-1}(Y_1, \ldots, Y_{n-1}) + \lambda Y_n}],$$ 
for explicitly determined polynomials $F_1, \ldots, F_{n-1}$.  Also, $\mc{V}_n$ is an $n$-fold extension of 
$\mc{G}^{(\lambda^p)}$'s, and as a scheme is given by
$$\Spec R[X_1, \ldots, X_n, \frac{1}{1+\lambda^p X_1}, \frac{1}{G_1(X_1)+\lambda^p X_2}, \ldots, \frac{1}{G_{n-1}(X_1, \ldots, X_{n-1}) + 
\lambda^p X_n}],$$ for explicitly determined polynomials $G_1, \ldots, G_{n-1}$.
\end{prop}

\begin{remark}\label{Rsscomplicated}
While there is an explicit algorithm to calculate the polynomials $F_i$ and $G_i$, the calculation gets extremely complicated as $n$ gets large.  Indeed,
for $n=3$, determining $G_2$ in \cite{SS:kasw2} takes five pages.  Likewise, the proof of Proposition \ref{Pkaswtheory} requires extraordinarily
complicated calculations.
\end{remark}  

\begin{remark}\label{Rss2gm}
There are morphisms $\mc{W}_n \to (\Gp_m)^n$ and $\mc{V}_n \to (\Gp_m)^n$, 
given by 
$$\alpha^{(n)}: (Y_1, \ldots, Y_n) \mapsto (1 + \lambda Y_1, F_1(Y_1) + \lambda Y_2, \ldots, F_{n-1}(Y_1, \ldots, Y_{n-1}) + \lambda Y_n)$$
and
$$\beta^{(n)}: (X_1, \ldots, X_n) \mapsto (1 + \lambda^p X_1, G_1(X_1) + \lambda^p X_2, \ldots, G_{n-1}(X_1, \ldots, X_{n-1}) + \lambda^p X_n),$$
respectively.  These maps are isomorphisms on the generic fibers, and the diagram
\begin{equation}\label{Essn}
\xymatrix{
\Spec \mc{W}_n \ar[r]^{\alpha^{(n)}} \ar[d]^{\psi_n} & (\Gp_m)^n\ar[d]^{\phi_n} \\
\Spec \mc{V}_n \ar[r]^{\beta^{(n)}} & (\Gp_m)^n
}
\end{equation}
commutes, where $\phi_n$ is given by (\ref{Egenkummer}) and $\psi_n$ is given by (\ref{Ekasw}).
\end{remark}

We state the generalization of Proposition \ref{Psszp}:

\begin{prop}[\cite{SS:kasw1}, Theorem 3.8 or \cite{SS:kasw2}, Theorem 3.8]\label{Psszpn}
Suppose $R$ contains the $p^n$th roots of unity.  If $B$ is a flat local $R$-algebra, then any unramified $\ints/p^n$-cover $\Spec C \to \Spec B$ is 
given by a Cartesian diagram of the form
\[
\xymatrix{
\Spec C \ar[r] \ar[d] & \mc{W}_n \ar[d]^{\psi_n} \\
\Spec B \ar[r] & \mc{V}_n = \mc{W}_n/(\ints/p^n)
}
\]
\end{prop}
\begin{proof} One shows that $H^1(\Spec B, \mc{W}_n) = 0$ (\cite[Theorem 3.6]{SS:kasw1}), and then the proof is the same as that of Proposition
\ref{Psszp}.
\end{proof}

\subsection{Lifting $\ints/p^2$-extensions}\label{Szp2lifting}
The aim of this section is to sketch the proof of the following theorem of Green and Matignon:

\begin{theorem}[\cite{GM:lg}, Theorem 2]\label{Tzp2lifting}
The group $\ints/p^2$ is a local Oort group.  Moreover, any $\ints/p^2$-extension of $k[[t]]$ lifts over $W(k)[\zeta_{p^2}]$.
\end{theorem}

Green and Matignon do this by making the diagram in Proposition \ref{Psszpn} more explicit in the case $n=2$.  
More specifically, they take $R = W(k)[\zeta_{p^2}]$ and $B = R[[T]]_{(\pi)}$, where $\pi$ is a uniformizer of $R$, 
and write down a specific morphism $\Spec B \to \mc{V}_2$.  By 
Proposition \ref{Psszpn}, this gives an unramified $\ints/p^2$-cover $\Spec C \to \Spec B$.  They then show that   
$C'/R[[T]]$, where $C'$ is the integral closure of $R[[T]]$ in $\Frac(C)$, reduces to a specific $\ints/p^2$-Galois extension $A/k[[t]]$
(in fact, it is an extension in the form of Proposition \ref{Ppnexplicit} where $x_2 = 0$).  
Thus  $C'/R[[T]]$ is a lift of $A/k[[t]]$.  Lastly, they show how, given a lift of $A/k[[t]]$, one can lift all other $\ints/p^2$-extensions of $k[[t]]$.

We start by constructing the map $\Spec C \to \Spec B$.
Let $\lambda = \zeta_p - 1$ where $\zeta_p$ is a nontrivial $p$th root of unity, let $\pi = \zeta_{p^2} - 1$, where $\zeta_{p^2}^p = 
\zeta_p$, and let $\mu = \pi - \pi^2/2 + \cdots + (-1)^p\pi^{p-1}/(p-1)$.  Note that $\mu$ and $\pi$ are both uniformizers of $R$.
Recall that $$\mc{W}_n \cong \Spec R[Y_1, Y_2, \frac{1}{1 + \lambda Y_1}, \frac{1}{F_1(Y_1) + \lambda Y_2}]$$ and
$$\mc{V}_n \cong \Spec R[X_1, X_2, \frac{1}{1 + \lambda^p X_1}, \frac{1}{G_1(X_1) + \lambda^p X_2}].$$
We have $F_1(Y_1) = \exp_p(\mu Y_1)$ and $G_1(X_1) = \exp_p(\mu^p X_1)$, where $\exp_p$ is the truncated exponential including terms up 
through degree $p-1$.  The formula for $F_1$ comes from \cite{SS:kasw1}, and Green and Matignon derive the formula for 
$G_1$, although it can also be determined from \cite{SS:kasw2}.  Moreover, by \cite[Theorem 7.1]{SS:kasw1}, there is a commutative diagram
of group schemes over $R$:
\[
\xymatrix{
\mc{W}_1 = \mc{G}^{(\lambda)} \ar[r] \ar[d]^{\psi_1} & \mc{W}_2 \ar[d]^{\psi_2} \ar[r]^{\alpha^{(2)}} & (\Gp_m)^2 \ar[d]^{\phi_2} \\
\mc{V}_1 = \mc{G}^{(\lambda^p)}  \ar[r] & \mc{V}_2 \ar[r]^{\beta^{(2)}} & (\Gp_m)^2,
}
\]
where the $\psi_i$'s are from Proposition \ref{Pkaswtheory}, the left horizontal arrows are given by $X_2 = 0$ and $Y_2 = 0$, the map $\phi_2$ is
from (\ref{Egenkummer}), the map $\alpha^{(2)}$ is given by $(Y_1, Y_2) \mapsto (1 + \lambda Y_1, F_1(Y_1) + \lambda Y_2)$, and
$\beta^{(2)}$ is given by $(X_1, X_2) \mapsto (1 + \lambda^p X_1, G_1(X_1) + \lambda^p X_2)$ (the right hand square is nothing but (\ref{Essn}), 
for $n=2$).  From this diagram, we see that the map
$\psi_2: \mc{W}_2 \to \mc{V}_2$ is given by the equations
\begin{eqnarray*}
(1 + \lambda Y_1)^p &=& 1 + \lambda^p X_1\\
(F_1(Y_1) + \lambda Y_2)^p(1 + \lambda Y_1)^{-1} &=& G_1(\psi_1(Y_1)) + \lambda^p X_2.
\end{eqnarray*} 
Note that $\psi_1$ is the same as $\psi$ from (\ref{Epsi}), that is, $\psi_1(Y_1) = \frac{(1 + \lambda Y_1)^p - 1}{\lambda^p}$.  
Let $\alpha: \Spec B \to \mc{V}_2$ be given by $X_1 \to T^{-u_1}$ and $X_2 \to 0$.  Then, in Proposition \ref{Psszpn}, we have
\begin{equation}\label{Ep2extension}
C \cong B[Y_1, Y_2]/(\psi_1(Y_1) - T^{-u_1}, (F_1(Y_1) + \lambda Y_2)^p - (1 + \lambda Y_1)G_1(T^{-u_1})),
\end{equation}
and $\Spec C \to \Spec B$ is an unramified cover, which must reduce to some $\ints/p^2$-cover of $k((t))$ if $t$ is the reduction of $T$.

An intricate calculation (\cite[Lemmas 5.2, 5.3]{GM:lg}) shows that (\ref{Ep2extension}) in fact reduces to the
the $\ints/p^2$-extension $A/k((t))$ in Proposition \ref{Ppnexplicit} with $x_1 = t^{-u_1}$ and $x_2 = 0$, which has upper jumps $(u_1, pu_1)$.  
By Lemma \ref{Ljumpsdifferent}, the 
degree of the different of this extension is $$(p^2 - 1)(u_1 + 1) + p(p-1)^2 u_1.$$  

Now, if we let $Z_1 = 1 + \lambda Y_1$ and $Z_2 = F_1(Y_1) + \lambda Y_2$, then we can write the equations for $C$ as 
\begin{equation}\label{Ep2general}
Z_1^p = 1 + \lambda^p T^{-u_1}, \ Z_2^p = Z_1 \cdot G_1(T^{-u_1}).
\end{equation}
Letting $C'$ be the normalization of $R[[T]]$ in $\Frac(C)$, we see that $\Spec (C' \otimes_R K) \to \Spec (R[[T]] \otimes_R K)$ is branched of 
order $p^2$ at $u_1 + 1$ points (the zeroes and poles of $1 + \lambda^p T^{-u_1}$ in the open unit disc around $T=0$) and of order $p$ at 
$(p-1)u_1$ points (the zeroes of $G_1(T^{-u_1})$, or $\exp_p(\mu^p T^{-u_1})$, in the open unit disc around $T=0$).  
Since $\cf(K) = 0$, the degree of the different of this extension is also $$(p^2-1)(u_1 + 1) + p(p-1)^2 u_1.$$  By the different 
criterion (Proposition \ref{Pdifferent}), $C'/R[[T]]$ is a lift of $A'/k[[t]]$, where $A'$ is the integral closure of $k[[t]]$ in $A$.

Green and Matignon then show (\cite[Lemma 5.4, Theorem 5.5]{GM:lg}) that $C$ can be deformed so as to find a lift of the integral closure of
$k[[t]]$ in the $\ints/p^2$-extension $A/k[[t]]$, corresponding to $x_1 = t^{-u_1}$ and $x_2$ \emph{arbitrary} in Proposition \ref{Ppnexplicit}.
This is a subtle calculation, and we remark that the na\"{i}ve approach (i.e., choosing $\alpha: \Spec B \to \mc{V}_2$ to send $X_2$ to a lift of $x_2$ 
instead of to $0$) fails because the different criterion is no longer satisfied in general.  We will discuss this further in \S\ref{Szpnlifting}.

\subsection{A different approach to lifting $\ints/p^n$-extensions}\label{Szpnlifting}
\subsubsection{The general form.}\label{Sgenform}
Since the explicit polynomials $F_n$ and $G_n$ involved in Sekiguchi-Suwa theory (Proposition \ref{Pkaswtheory}) get very complicated when $n$ gets
large, it is difficult to generalize the methods of Green and Matignon to show that $\ints/p^n$ is a local Oort group.  In (as of yet unpublished) work of
Stefan Wewers and the author, a different, less explicit approach is taken.  The starting point is equation (\ref{Ep2general}).  Generalizing this,
it is clear that if $R$ contains the $p^n$th roots of unity, then any $\ints/p^n$-extension of $\Frac(R[[T]])$ can be given by equations
\begin{equation}\label{Egenform}
\begin{split}
Z_1^p &= H_1(T) \\
Z_2^p &= Z_1 H_2(T) \\
&\vdots \\
Z_n^p &= Z_{n-1} H_n(T), 
\end{split}
\end{equation}
where $H_i(T) \in \Frac(R[[T]])$ for $1 \leq i \leq n$.  The Galois action of a generator of $\ints/p^n$ sends $Z_i$ to $\zeta_{p^i}Z_i$, where
$\zeta_{p^i}$ is a $p^i$th root of unity and $\zeta_{p^j}^{p^{j-i}} = \zeta_{p^i}$ for $j \geq i$.
This is isomorphic to the extension $Z^{p^n} = H_1(T)H_2(T)^p \cdots H_n(T)^{p^{n-1}}$ in standard Kummer form.

The next proposition gives a somewhat explicit criterion for when a birational lift of a cyclic extension is an actual lift.

\begin{prop}\label{Pbirationalactual}
\begin{enumerate}[(i)]
\item
Let $A/k[[t]]$ be a $\ints/p^n$-extension with upper jumps $(u_1, \ldots, u_n)$.  Suppose $A_R/R[[T]]$ is a birational lift of $A/k[[t]]$ given
by normalizing $R[[T]]$ in the extension of fraction fields given by 
(\ref{Egenform}).  If the $H_i$ are polynomials in $T^{-1}$, if $\deg H_1 = u_1$, and if $\deg H_i = u_i - u_{i-1}$ for $i > 1$, then 
$A_R/R[[T]]$ is a lift of $A/k[[t]]$.
\item Suppose that $A_R/R[[T]]$ is the normalization of $R[[T]]$ in the extension of fraction fields generated by (\ref{Egenform}), 
where the $H_i$ are polynomials in $T^{-1}$ with $\deg H_1 = u_1$, $\deg H_i = u_i - u_{i-1}$ for $2 \leq i \leq n-1$, and 
$\deg H_n = (p-1)u_{n-1}$.  
Suppose further that the $\ints/p^{n-1}$-subextension $B_R/R[[T]]$ of $A_R/R[[T]]$ reduces to a $\ints/p^{n-1}$-extension $B/k[[t]]$ with upper jumps 
$(u_1, \ldots, u_{n-1})$. 
If the reduction $A/k[[t]]$ of $A_R/R[[T]]$ to characteristic $p$ gives a separable extension of fraction fields, then $A$ is integrally closed and the 
upper jumps of $A/k[[t]]$ are $(u_1, \ldots, u_{n-1}, pu_{n-1})$.  Thus $A_R/R[[T]]$ is a lift of \emph{some} $\ints/p^n$-extension $A/k[[t]]$.
\end{enumerate}
\end{prop}

\begin{proof}
\emph{To (i):}
If $T=x$ is in the open
unit disc $\mc{D}$, and $i$ is minimal such that $T=x$ is a zero of $H_i(T^{-1})$, then the branching index of $T = x$ in 
$\Spec (A_R \otimes_R K) \to \Spec (R[[T]] \otimes_R K)$ is at most $p^{n-i+1}$.  
Clearly $T = 0$ is the only pole of any $H_i(T^{-1})$, and has branching index at most $p^n$.  
Thus, the zeroes of $H_1(T^{-1})$ in $\mc{D}$, as well as $T=0$, are branched of index at most $p^n$.
The zeroes of $H_2(T^{-1})$ that are not zeroes of $H_1(T^{-1})$ are branched of index at most $p^{n-1}$, etc.  Since $H_i$ has at most $\deg H_i$ 
zeroes, it follows that the degree
$\delta_{\eta}$ of the different of $(A_R \otimes_R K)/(R[[T]] \otimes_R K)$ satisfies

\begin{equation}\label{Egenericdiff1}
\delta_{\eta} \leq (p^n - 1)(u_1 + 1) + \sum_{i=2}^n (p^{n-i+1} - 1)p^{i-1}(u_i - u_{i-1}).
\end{equation}
By Lemma \ref{Ljumpsdifferent}, the right-hand side is the degree $\delta_s$ of the different of $A/k[[t]]$.
By the different criterion (Proposition \ref{Pdifferent}), we have $\delta_{\eta} = \delta_s$, and our birational lift is an actual lift.
\\
\\
\emph{To (ii):} As in part (i), the degree $\delta_{\eta}$ of the different of $(A_R \otimes_R K)/(R[[T]] \otimes_R K)$ satisfies
\begin{equation}\label{Egenericdiff2}
\delta_{\eta} \leq (p^n - 1)(u_1 + 1) + \left(\sum_{i=2}^{n-1} (p^{n-i+1} - 1)p^{i-1}(u_i - u_{i-1})\right) + (p-1)p^{n-1}((p-1)u_{n-1}).
\end{equation}
Let $(u_1, \ldots, u_n)$ be the upper jumps of $A'/k[[t]]$, 
where $A'$ is the integral closure of $A$.  Then the degree of the different of
$A'/k[[t]]$ is $$\delta_s' := (p^n - 1)(u_1 + 1) + \sum_{i=2}^n (p^{n-i+1} - 1)p^{i-1}(u_i - u_{i-1}),$$ by Lemma \ref{Ljumpsdifferent}.  
If $\delta_s$ is the degree of
the different of $A/k[[t]]$, then $\delta_s' \leq \delta_s = \delta_{\eta}$.  But, by Corollary \ref{Ccomparejumps}, we have $\delta_s' \geq \delta_{\eta}$,
with equality iff $u_n = pu_{n-1}$.  So $\delta_s = \delta_s'$, which means that $A = A'$ and the $u_i$ are as desired.
\end{proof}

\begin{remark}\label{Requality}
Since $\delta_{\eta} = \delta_s$ in both cases above, we have equalities in (\ref{Egenericdiff1}) and (\ref{Egenericdiff2}), 
so each of the zeroes of each $H_i(T^{-1})$ must be simple and must lie in the open unit 
disc around $T=0$.  This means that, up to scaling by a constant, each polynomial $H_i(T^{-1})$ is of the form
$a_0 + a_1T^{-1} + \cdots + a_{N_i}T^{-{N_i}}$ with $v(a_0) = 0$ and $v(a_j) > 0$ for $j > 0$.  Furthermore, the zeroes of all of the
$H_i(T^{-1})$ must be pairwise distinct. 
\end{remark}

Note that Propostion \ref{Pbirationalactual}(ii) applies to Green and Matignon's lift in (\ref{Ep2general}), as $G_1(\psi_1(Y_1)) = 
\exp_p(\mu^pT^{-u_1})$.  In fact, Proposition \ref{Pbirationalactual}(i)
applies to Green and Matignon's lifts for general $\ints/p^2$-extensions (this is how they prove that the birational formulas they write down are actually 
lifts).  So it seems reasonable, given an $\ints/p^n$-extension, to look for a lift in the form of (\ref{Egenform}) satisfying the conditions of Proposition 
\ref{Pbirationalactual}(i).

Furthermore, Proposition \ref{Pbirationalactual} suggests a framework to attack the lifting problem for cyclic extensions:  

\begin{framework}\label{Fproof}
Suppose 
$A/k[[t]]$ is a $\ints/p^n$-extension with upper jumps $(u_1, \ldots, u_n)$.  Let $B/k[[t]]$ be the $\ints/p^{n-1}$-subextension.  Assume by induction
that we have a lift $B_R/R[[T]]$ of $B/k[[t]]$ given by equations in the form of (\ref{Egenform}), where the $H_i$ are polynomials in $T^{-1}$ 
with $\deg H_1 = u_1$, and $\deg H_i = u_i - u_{i-1}$ for $2 \leq i \leq n-1$.   

\emph{Step 1:} We seek a polynomial $H_n$ in $T^{-1}$ of degree $(p-1)u_{n-1}$ so that the reduction of (\ref{Egenform}) is separable.
By Proposition \ref{Pbirationalactual}(ii), the equations (\ref{Egenform}) then lift \emph{some} 
$\ints/p^n$-extension with upper jumps $(u_1, \ldots, u_{n-1}, pu_{n-1})$ whose $\ints/p^{n-1}$-subextension is $B/k[[t]]$.  

\emph{Step 2:} If
we can find such an $H_n$, then the goal is to replace $H_n$ by a polynomial of degree $u_n - u_{n-1}$ in order to give a birational
lift of $A/k[[t]]$, using an argument along the lines of that of Green and Matignon.  This will be an actual lift by Proposition \ref{Pbirationalactual}.
\end{framework}

\subsubsection{Depth and separability.}\label{Sdepthandsep}
Framework \ref{Fproof} forces us to understand when an extension of $R[[T]]$ has separable reduction.  In Appendix \ref{Sstablegraph}, the notion
of the \emph{depth} of a cyclic $\ints/p^n$-extension $A_R/R[[T]]$ is discussed.  This is a non-negative number that is $0$ iff the
extension has separable reduction.  

Fix a $\ints/p^n$-extension $A_R/R[[T]],$ corresponding to a morphism $f_R: \Spec A_R \to \Spec R[[T]]$, birationally given by equations in the form 
(\ref{Egenform}).  For each $i$, $1 \leq i \leq n$, let $(A_R)_i/R[[T]]$ be the unique $\ints/p^i$-subextension, and let $(f_R)_i$ be the corresponding
morphism.  
For each rational $r \geq 0$, let $a_r$ be an element of a finite extension of $R$ of valuation $r$, and enlarge $R$ so that it contains $a_r$.  
Then $\mc{D}_r := \Spec R[[a_r^{-1}T]]$ is the open disc of
radius $|a_r|$ (\S\ref{Sdiscs}).  Note that $\mc{D} = \mc{D}_0$.  The restriction of any $(f_R)_i$ to the disc $\mc{D}_r$ corresponds to making a 
substitution $U = a_r^{-1}T$ in the $H_1(T^{-1}), \ldots, H_i(T^{-1})$, and viewing the equations (\ref{Egenform}) as giving an extension of 
$R[[U]]$.  We denote the depth
of such an extension by $\delta_i(r)$, and we view $\delta_i$ as a function of $r$.  Our goal is to find $H_1, \ldots, H_n$ such that
$\delta_n(0) = 0$.

The following result is unpublished, but is known to the experts.  It can essentially be derived from \cite[\S5.3]{Ob:fm1}, where the depth 
is called the ``effective different."
\begin{lemma}\label{Ldepthslope}
Let $(f_R)_i$ and $a_r$ be as above.
For each $1 \leq i \leq n$, the depth $\delta_i$ is a piecewise-linear function of $r$.  The right-derivative of $\delta_i$ at $r$ is less than or
equal to $\nu_i(r) - 1$, where $\nu_i(r)$ is the number of branch points of the generic fiber of $(f_R)_i$ with valuation greater than $r$ (in 
terms of the coordinate $T$).  
Equality holds iff the special fiber of $(f_R)_i\left|_{\mc{D}_r} \right.$ is smooth (i.e., the integral closure of $R[[a_r^{-1}T]]$ in 
$\Frac((A_R)_i)$ has integrally closed reduction).
\end{lemma}

Let us now place ourselves in Framework \ref{Fproof} and use the notation therein.  We will make the further assumption (for the remainder of
\S\ref{Sdepthandsep}) that the zeroes of
$H_1, \ldots, H_{n-1}$ (and thus the branch points of the generic fiber of $(f_R)_{n-1}$) all have valuation greater than $\frac{1}{u_{n-1}(p-1)}$.  
This assumption is essential (in fact, its necessity is the reason that we do not have a full proof that $\ints/p^n$ is a local Oort group), as it has the 
consequence that $\delta_{n-1}(\frac{1}{u_{n-1}(p-1)}) = \frac{1}{p-1}$.  This follows from Lemma
\ref{Ldepthslope}.  Then, \cite[Proposition 6.1(iii)]{We:fr} shows that, no matter what we pick for $H_n$, we will have 
$\delta_{n}(\frac{1}{u_{n-1}(p-1)}) = \frac{p}{p-1}$ (here $\delta_{n-1}$ and $\delta_n$ play the roles of $\ol{\delta}$ and $\delta$ in \cite{We:fr}).  

\begin{lemma}\label{Lsepsmooth}
If we pick $H_n$ of degree $(p-1)u_{n-1}$ in $T^{-1}$ (as in Framework \ref{Fproof})
so that all of its zeroes are simple and have valuation $\frac{1}{u_{n-1}(p-1)}$, then $\delta_n(0) = 0$ iff the special fiber of $f_R$ is smooth. 
\end{lemma}

\begin{proof} 
By Remark \ref{Requality}, we see that $f_R$ has a total of $pu_{n-1} + 1$ branch points on the 
generic fiber (corresponding to the zeroes of $H_1(T^{-1}), \ldots, H_n(T^{-1})$, which are all simple and distinct, and the unique shared pole $T=0$).
Since $\delta_n(\frac{1}{u_{n-1}(p-1)}) = \frac{p}{p-1}$, Lemma \ref{Ldepthslope} shows that $\delta_n(0) = 0$ iff the special fiber of 
$f_R$ is smooth.  
\end{proof}

The next proposition is a major result in \cite{OW:oc}.

\begin{prop}\label{Psmoothing}
Given $H_n$ as in Lemma \ref{Lsepsmooth}, we can adjust $H_n$, while keeping its degree and the valuation of its zeroes the same, so that
the special fiber of $f_R$ becomes smooth.
\end{prop}

\begin{proof}[Very rough sketch of proof] If $f_R$ is not smooth, there is a minimum $r \geq 0$ corresponding to a maximal disk $\mc{D}_r$
such that $(f_R)\left|_{\mc{D}_r}\right.$ is smooth.  One can show using 
\cite[Proposition 6.1]{We:fr} that $r \leq \frac{1}{u_{n-1}(p-1)}$.  One then shows that if $r > 0$ and $R$ is taken large enough, 
then one can adjust $H_n$ as in the proposition to decrease $r$ (we think of $r$ as depending on $H_n$).  This involves solving differential equations
over $k$, which in turn reduces to solving systems of linear equations over $\FF_p$.  Amazingly, these systems are overdetermined yet
still always have solutions!  
Lastly, one independently shows that there is some particular $H_n$ as in the proposition where $r$ takes a minimum (and $H_n$ can be realized
over some \emph{finite} extension $R/W(k)$).  This minimum must be $0$, and this is the $H_n$ we seek.
\end{proof}

\begin{remark}\label{Rnotexplicit}
Note that this proof is non-constructive.  In particular, we have no control over how large we must take $R$.
\end{remark}
\begin{corollary}\label{Csmoothing}
Suppose $B/k[[t]]$ is a $\ints/p^{n-1}$-extension with upper jumps $(u_1, \ldots, u_{n-1})$ that lifts to characteristic zero, and that the
lift to characteristic zero is in the form of (\ref{Egenform}) with $H_1, \ldots, H_{n-1}$ polynomials in $T^{-1}$ with $\deg H_1 = u_1$ and
$\deg H_i = u_i - u_{i-1}$ for $2 \leq i \leq n-1$.  Suppose further that all the zeroes of 
$H_1(T^{-1}), \ldots, H_{n-1}(T^{-1})$ have valuation
greater than $\frac{1}{u_{n-1}(p-1)}$.  Then there is a $\ints/p^n$-extension
$A_0/k[[t]]$ with upper jumps $(u_1, \ldots, u_{n-1}, pu_{n-1})$ and a $\ints/p^{n-1}$-subextension $B/k[[t]]$ that lifts to characteristic zero.
\end{corollary}

\begin{proof}
Follows from Step 1 of Framework \ref{Fproof}, Lemma \ref{Lsepsmooth}, and Proposition \ref{Psmoothing}.
\end{proof}

\subsubsection{Lifting general extensions.}\label{Sgeneral}
By Corollary \ref{Csmoothing}, we can, under mild hypotheses, extend a lift of a $\ints/p^{n-1}$-extension $B/k[[t]]$ 
to a lift of some $\ints/p^n$-extension $A_0/k[[t]]$ with minimal possible $n$th upper jump.  
The extension $A_0/k[[t]]$ can be written in the form of Proposition \ref{Ppnexplicit}, where $\deg x_n < pu_{n-1}$.  Its lift can be written in the form of 
(\ref{Egenform}) with $H_n$ a polynomial of degree $(p-1)u_{n-1}$ in $T^{-1}$.  
Our goal is now to show that we can, in fact, extend our lift of $B/k[[t]]$ to lift \emph{any} $\ints/p^n$-extension,
$A/k[[t]]$ (with $B/k[[t]]$ as the $\ints/p^{n-1}$-subextension) no matter what the upper jump.  In particular, if the $n$th upper jump of $A/k[[t]]$
is $u_n$, we wish to replace $H_n$ with a polynomial of degree $u_n - u_{n-1}$ in $T^{-1}$ that yields a birational lift of $A/k[[t]]$.  By Proposition
\ref{Pbirationalactual} (i), this gives an actual lift of $A/k[[t]]$.

Let us recall that our lift of the $\ints/p$-extension given birationally by $y^p - y = t^{-u_1}$ was given by the equation
$Z^p = 1 + \lambda^pT^{-u_1}$ (\S\ref{Szplifting}).  Now, suppose we were to replace $t^{-u_1}$ by $t^{-u_1} + q(t^{-1})$, where $q$  
has no terms of degree divisible by $p$.
If $Q(T^{-1})$ is any lift of $q(t^{-1})$ to $R[T^{-1}]$ that preserves the degree, the argument of Theorem \ref{Tzplifting} easily carries
through to show that $Z^p = 1 + \lambda^p (T^{-u_1} + Q(T^{-1}))$ gives a lift of the extension given by $y^p - y = t^{-u_1} + q(t^{-1})$.

The general situation is similar.  In fact, in \cite{OW:oc}, we prove the following result: 

\begin{prop}\label{Preplace}
Suppose $A/k[[t]]$ is obtained from $A_0/k[[t]]$ as above by replacing $x_n$ in 
Proposition \ref{Ppnexplicit} with $x_n + q(t^{-1})$, where $q$ has no terms of degree divisible by $p$.  
If $Q(T^{-1})$ is any lift of $q(t^{-1})$ to $R[T^{-1}]$, then replacing $H_n$ by 
$H_n + \lambda^pQ(T^{-1})$ in (\ref{Egenform}) gives a birational lift of $A/k[[t]]$.
\end{prop}

Unfortunately, we cannot apply Proposition \ref{Pbirationalactual}(i) in general.  For instance, if $\deg q > pu_{n-1}$, then $\deg q = u_n$, thus
$\deg Q \geq u_n$ and $\deg (H_n + \lambda^pQ) \geq u_n$.  But to apply Proposition \ref{Pbirationalactual}(i), we would need 
$\deg (H_n + \lambda^p Q) = u_n - u_{n-1}$.  However, it turns out that there is some polynomial $H_n'$ of degree $u_n - u_{n-1}$ in $T^{-1}$,
as well as some lift $Q$ of $q$ to $R[[T]]$, such that the ratio of $H_n'$ to $H_n + \lambda^p Q$ is a $p$th power in $\Frac(R[[T]])$.
Replacing $H_n$ by $H_n'$ must give the same extension of $\Frac(R[[T]])$ as replacing $H_n$ by $H_n + \lambda^p Q$.  By Proposition 
\ref{Preplace}, replacing $H_n$ by $H_n'$ gives a birational lift, and by Proposition \ref{Pbirationalactual}(i), this is an actual lift.  Thus
we have completed step $2$ of Framework \ref{Fproof}.  Explicitly:

\begin{prop}\label{Pgeneralextend}
Suppose $B/k[[t]]$ is a $\ints/p^{n-1}$-extension with upper jumps $(u_1, \ldots, u_{n-1})$ that lifts to characteristic zero, and that the
lift to characteristic zero is in the form of (\ref{Egenform}) with $H_1, \ldots, H_{n-1}$ satisfying the assumptions of Corollary 
\ref{Csmoothing}.  Then any $\ints/p^n$-extension
$A/k[[t]]$ with $\ints/p^{n-1}$-subextension $B/k[[t]]$ lifts to characteristic zero.
\end{prop}

However, in order to continue the induction, it is necessary that the zeroes of $H_n'(T^{-1})$ have valuation greater than
$\frac{1}{u_n(p-1)}$.  In our construction of $H_n'$, this is regrettably not always the case.  The best we can do so far is the following:

\begin{lemma}\label{Lbounds}
Suppose that there is no $a \in p\ints$ such that $u_n - pu_{n-1} < a \leq (u_n - pu_{n-1})\left(\frac{u_n}{u_n - u_{n-1}}\right)$.
Then the lift of $A/k[[t]]$ in Proposition \ref{Pgeneralextend} can be accomplished in the form of (\ref{Egenform}) with $H_n$ a polynomial
in $T^{-1}$ of degree $u_n - u_{n-1}$ and the zeroes of $H_n(T^{-1})$ having valuation greater than $\frac{1}{u_n(p-1)}$.
\end{lemma}

Using Lemma \ref{Lbounds}, we can state our strongest result about the local lifting problem for $\ints/p^n$.

\begin{theorem}[\cite{OW:oc}]\label{Tzpnlifting}
Let $A/k[[t]]$ be a $\ints/p^n$-extension with upper jumps $(u_1, \ldots, u_n)$.  Then $A/k[[t]]$ lifts to characteristic zero 
so long as, for each $i$, $3 \leq i \leq n-1$, there does not exist $a_i \in p\ints$ with
$$u_i - pu_{i-1} < a_i \leq (u_i - pu_{i-1})\left(\frac{u_i}{u_i - u_{i-1}}\right).$$
\end{theorem}

\begin{proof}
The cases $n=1$ and $n=2$ are given by Theorems \ref{Tzplifting} and \ref{Tzp2lifting}.  Furthermore, the lifts
given explicitly in Theorem \ref{Tzplifting} and by Green and Matignon for Theorem \ref{Tzp2lifting} are in the form of (\ref{Egenform}) where 
$\deg H_1 = u_1$, $\deg H_2 = u_2 - u_1$, and the zeroes of $H_1(T^{-1})$ and $H_2(T^{-1})$ have valuation greater than $\frac{1}{u_1(p-1)}$ 
and $\frac{1}{u_2(p-1)}$, respectively.  
The proof then follows by repeatedly applying Proposition \ref{Pgeneralextend} and Lemma \ref{Lbounds}
\end{proof}

\begin{remark}\label{Rcyclic}
\begin{enumerate}[(i)]
\item The condition in Theorem \ref{Tzpnlifting} is vacuous for $n=3$, so $\ints/p^3$ is a local Oort group.
\item The condition is satisfied if $u_i = pu_{i-1}$ for all $i \geq 3$.  In particular, it holds for all ``minimal" sets of upper jumps
$(u_1, pu_1, \ldots, p^{n-1}u_1)$.
\item The example $p=5$, $n=4$, $(u_1, \ldots, u_4) = (1, 5, 34, 170)$, and $a_3 = 10$ shows that the condition is not always satisfied. 
\item The condition is somewhat strange: for a $\ints/p^4$-extension, the only upper jumps that need to be checked are the second and third!
Because of this strangeness, it seems reasonable to believe that the condition is not necessary, and that the Oort conjecture still holds.
\item Our proof gives no insight into the ring $R$ necessary for lifting to characteristic zero.  To show that a lift is possible over
$R = W(k)[\zeta_{p^n}]$ will require new techniques.
\end{enumerate}
\end{remark}

\section{Metacyclic groups}\label{Smetacyclic}
Throughout \S\ref{Smetacyclic}, $G \cong \ints/p \rtimes \ints/m$ with $p \nmid m$ and $G$ \emph{not} cyclic.
Also, $R/W(k)$ is a large enough finite extension and $K = \Frac(R)$.  Let $\mc{D} = \Spec R[[T]]$ be the open unit disc.  
Bouw, Wewers, and Zapponi have proven the following theorem (although they did not phrase it in terms of the KGB obstruction).

\begin{theorem}[\cite{BWZ:dd}, Theorem 2.1]\label{Tzpzm}
A $G$-extension $A/k[[t]]$ lifts to characteristic zero iff its KGB obstruction vanishes.  In other words, $A/k[[t]]$ lifts exactly when the 
(unique) positive jump in the higher ramification filtration for the lower numbering of $A/k[[t]]$ is congruent to $-1$ (mod $m$).
\end{theorem}

The following corollary immediately follows from Theorems \ref{TKGB} and \ref{Tzpzm}.

\begin{corollary}\label{Czpz2}
For odd primes $p$, the dihedral group $D_p$ is a local Oort group.
\end{corollary}

Note that, as a consequence of Proposition \ref{Pzpzm}, the KGB obstruction can only vanish if the conjugation action of 
$\ints/m$ on $\ints/p$ is faithful.  In particular, we must have $m | (p-1)$.  In light of this, we view the $m$th roots of unity as 
living in $(\ints/p)^{\times} \subseteq k^{\times}$.

\subsection{An explicit lifting example}\label{Sexplicitzpzm}
We present an example of a $G$-extension that lifts to characteristic zero, due to Green and Matignon.

\begin{prop}[\cite{GM:op}, IV Prop.\ 2.2.1]\label{Pdihedralexample} 
Let $p$ be an odd prime.  The $D_p$-extension of $k[[t]]$ given birationally by the equations 
\begin{equation}\label{Edihedral0}
x^2 = t, \  y^p - y = x^{-1},
\end{equation} 
lifts to characteristic zero (in fact, one can take $R = W(k)[\zeta]$, where $\zeta$ is a primitive $p$th root of unity). 
\end{prop}

\begin{proof}
There are elements $\sigma, c \in D_p$ with 
$$\sigma(x) = x,\ \sigma(y) = y+1,\ c(x) = -x, \ c(y) = -y.$$
Let $\zeta_p$ be a primitive $p$th root of unity, and let $\lambda = \zeta_p - 1$.
We claim the extension of $R[[T]]$ given birationally by equations 
\begin{equation}\label{Edihedral2}
(X + \frac{\lambda^p}{2})^2 = T,\ Z^p = 1 + \lambda^p X^{-1}
\end{equation}
gives a lift, under the automorphisms 
$$\sigma(X) = X, \ \sigma(Z) = \zeta_p Z, \ c(X) = -X - \lambda^p, \ c(Z) = \frac{1}{Z}.$$  

A straightforward calculation verifies that the given actions of $\sigma$ and $c$ indeed preserve the equations (\ref{Edihedral2}).  It is also clear
that the actions of $\sigma$ and $c$ generate a dihedral group.  Upon making the substitution $1 + \lambda Y = Z$, we see that the equations
(\ref{Edihedral2}) reduce to the equations (\ref{Edihedral0}), as in the proof of Theorem \ref{Tzplifting}.  
Since $c(Y) = \frac{-Y}{1 + \lambda Y}$ and $\sigma(Y) = \zeta Y + 1$, we see that the action of $D_p$ in characteristic
zero reduces to the action of $D_p$ in characteristic $p$.  Lastly, one checks that the unique positive lower jump on the special fiber is $1$ 
(Proposition \ref{Ppnexplicit}), so the degree of the different in characteristic $p$ is $3p-2$ by (\ref{Ebasicdifferent}). 
The generic fiber is branched at $T = (\frac{\lambda^p}{2})^2$ of index $p$ and at $T = 0$ of index $2$, so the degree of the different is also 
$3p-2$.  By the different criterion (Proposition \ref{Pdifferent}), the equations (\ref{Edihedral2}) give a lift.
\end{proof}

\begin{remark}\label{Rmetacyclicexamples}
Similar examples are given in \cite[Proposition 2.2.2]{GM:op} of $\ints/p \rtimes \ints/(p-1)$-extensions that lift, for any $p$.
In general, however, it is not known how to find explicit equations for lifting $G$-extensions, where $G \cong \ints/p \rtimes \ints/m$.
\end{remark}

\subsection{Structure theorems}\label{Sstructure}
Our first step toward the proof of Theorem \ref{Tzpzm} is a structure result completely classifying $G$-extensions of $k((t))$.

\begin{lemma}[\cite{Pr:cw3}, Lemma 2.1.2, \cite{Pr:fc}, Lemma 1.4.1 (iv)]\label{Lzpzmform}
Suppose $L/k((t))$ is a $G = \ints/p \rtimes \ints/m$-extension with unique positive lower jump $h$.  Then, up to a change in the parameter $t$, 
we have $$L \cong k((t))[u, y]/(u^m - t,\ y^p - y - u^{-h}).$$  Furthermore, given a primitive $m$th root of unity $\zeta_m$ 
(viewed as an element of $(\ints/p)^{\times}$), we can choose elements $\sigma \in G$ of order $p$ and $c \in G$ of order $m$ such that 
the Galois action is given by $$\sigma(u) = u,\ \sigma(y) = y+1,\  c(u) = \zeta_m u,\  c(y) = \zeta_m^{-h} y.$$  Also, 
we have $c\sigma c^{-1} = \sigma^{\zeta_m^{h}}$.
\end{lemma}
In particular, up to isomorphism, there is only one $G$-extension of $k((t))$ with given positive lower jump.

A more general structure theorem, this one in characteristic zero, will be useful for the proof of Theorem \ref{Tzpzm}.  In particular, we can 
classify certain $G$-extensions $B/R[[T]]\{T^{-1}\}$.  Note that $R[[T]]\{T^{-1}\}$ 
is a complete discrete valuation ring with residue field $k((t))$, and if $\pi$ is a uniformizer of 
$R$, then $\pi$ is also a uniformizer of $R[[T]]\{T^{-1}\}$.  As we will see, we only care about such extensions for which $B$ is abstractly isomorphic to
$R[[U]]\{U^{-1}\}$ (these correspond to extensions in which $R$ is algebraically closed).  Fix such an isomorphism.

We need three invariants to classify such extensions.  The first is the \emph{different} of the extension.  For the second, let $\sigma \in G$ be 
an order $p$ element.  Then $\frac{\sigma(U)}{U} - 1$ can be written as $\pi^n f(U)$, where $f(U)$ has coefficients of nonnegative valuation and at
least one coefficient of valuation zero.  The \emph{conductor} of $B/R[[T]]\{T^{-1}\}$ is the valuation of the reduction of $f(U)$ to $k((u))$ 
(where $u$ has valuation $1$).  Lastly, if $c \in G$ generates a subgroup $C$ of order $m$, then 
$c(T) \equiv \alpha T \pmod{T^2}$, where $\alpha$ reduces to a nonzero element $\ol{\alpha} \in k^{\times}$.  
This gives a character $C \to k^{\times}$, which is called the \emph{tame inertia character}.

\begin{lemma}[\cite{BW:ll}, \S2, in particular Proposition 2.3]\label{Lboundarystructure} 
The different, conductor, and tame inertia character of $G$-extension of $R[[T]]\{T^{-1}\}$ are well-defined.
Up to a change in the parameter $T$, there is at most one $G$-extension $R[[U]]\{U^{-1}\}/R[[T]]\{T^{-1}\}$ with a given different, conductor, and
tame inertia character.  This works even if $G$ is cyclic.
\end{lemma}

A $G$-extension of $R[[T]]\{T^{-1}\}$ with degree of different $\delta$, conductor $h$, and tame inertia character $\lambda$ is said to be
\emph{of type $(\delta, h, \lambda)$}.

\begin{remark}\label{Rneedp}
Unfortunately, Lemmas \ref{Lzpzmform} and \ref{Lboundarystructure} do \emph{not} hold if $\ints/p \rtimes \ints/m$ is replaced by 
$\ints/p^n \rtimes \ints/m$, for $n > 1$.  See Remark \ref{Rbighurwitz}.
\end{remark}

\subsubsection{Outline of proof.}\label{Szpzmoutline}

We outline the proof of Theorem \ref{Tzpzm}, which will depend on results from \S\ref{Shurwitz}.

\begin{proof}[Proof of Theorem \ref{Tzpzm}]
Recall from \S\ref{Sgalois} 
that if $A_R/R[[T]]$ is a lift of the $G$-extension $A/k[[t]]$, then $A_R$ is abstractly isomorphic to $R[[U]]$, the ring of functions on an 
open unit disc.  Furthermore, for any faithful 
$G$-action on $R[[U]]$ by continuous $R$-linear automorphisms (that is, a $G$-action on the open unit disc), 
the fixed ring $R[[U]]^G$ is abstractly isomorphic to $R[[T]]$.
In Proposition \ref{Paction2hurwitz} (using Definition \ref{Dgeneralhurwitz}), we show that a faithful 
$G$-action on the open unit disc $\Spec R[[U]]$ with separable reduction and positive conductor $h$ at the boundary gives rise to a so-called \emph
{Hurwitz
tree with conductor $h$}.  In Proposition \ref{Phurwitz2action}, we reverse this process, showing how a 
Hurwitz tree with conductor $h$ gives rise to a $G$-action on the open unit disc.  By Corollary \ref{Churwitz2action}, the $G$-extension 
$R[[U]]/R[[T]]$ associated to such a $G$-action reduces to a $G$-extension $A_h/k[[t]]$ with unique positive lower jump $h$.

In Proposition \ref{Phurwitzconstruction}, we construct, for each $h \equiv -1 \pmod{m}$, a 
Hurwitz tree of conductor $h$, thus yielding a lift of a $G$-extension $A_h/k[[t]]$ with unique positive lower jump $h$.
Note that any $G$-extension $A/k[[t]]$ with positive lower jump 
$h$ is isomorphic to $A_h/k[[t]]$ (Lemma \ref{Lzpzmform}). 
Since we have constructed a lift of $A_h/k[[t]]$, there
must also be a lift of $A/k[[t]]$.  This completes the proof of Theorem \ref{Tzpzm}.
\end{proof}

\subsection{Hurwitz trees}\label{Shurwitz}
Suppose $G$ acts faithfully on the open unit disc $\mc{D} = \Spec R[[U]]$ by $R$-automorphisms with no inertia above
a uniformizer of $R$ (that is, the action reduces to a faithful $G$-action on $\Spec k[[u]]$).  
The idea of a Hurwitz tree is to break the disc $\mc{D}$ up into smaller discs, punctured discs, and annuli upon which $G$ acts, to isolate the 
salient features of these actions in combinatorial form, and to keep track of how they connect with each other.  Henrio gave the first major
exposition of Hurwitz trees (\cite{He:ht}), dealing with the case of a $\ints/p$-action on an open disc.  The concept was extended by Bouw and
Wewers to encompass $G$-actions.  For the basic facts about discs 
and annuli that we will use, see Appendix \ref{Sdiscs}.  Our concept of Hurwitz tree is more restrictive than that of \cite{BW:ll}, 
in that what we call a Hurwitz tree is called a ``Hurwitz tree of different $0$" in \cite{BW:ll}. 

Throughout \S\ref{Shurwitz}, let $\mc{D}_K$ be the generic fiber of $\mc{D}$.  Clearly, $G$ acts on $\mc{D}_K$.
Let $\sigma$ be an element of $G$ of order $p$.  \\

\subsubsection{The Hurwitz tree associated to a $G$-action on the open unit disc.}\label{Shurwitzaction}

Let $\del \mc{D} = \Spec R[[U]]\{U^{-1}\}$ be the boundary of $\mc{D}$.  Then a $G$-action on $\mc{D}$ gives a $G$-action on
$\del \mc{D}$, with some conductor $h$.
Then \cite[I, Claim 3.3]{GM:lg} shows that the number of (geometric)
fixed points of the action of $\sigma$ on $\mc{D}_K$ is $h+1$.  We may assume that $R$ (and $K$) are large enough so that the fixed points 
$y_1, \ldots, y_{h+1}$ of $\sigma$ are all defined over $K$.  We assume further that $h > 0$, so there are at least two fixed points. 
Lastly, we assume that $G$ acts without inertia above a uniformizer of $R$.
Now, let $Y^{st}$ be the stable model of $\proj^1_K$ corresponding to the marked disc
$(\mc{D}; y_1, \ldots, y_{h+1})$ (see \S\ref{Smarkeddiscs}---this is essentially the minimal stable model of $\proj^1_K$ 
that separates the specializations of the $y_i$ and $\infty$, where $\mc{D}_K \subseteq \proj^1_K$ is viewed as being the open unit disc centered 
at $0$).  Let $\ol{Y}$ be the special fiber of $Y^{st}$.
If $\ol{\infty}$ is the specialization of the point $\infty \in \proj^1_K$ to $\ol{Y}$, then let $\ol{V} = \ol{Y} \backslash \{\ol{\infty}\}$.  The set of points
of $\proj^1_K$ that specialize to $\ol{V}$ is a closed disc $\mc{E}$, strictly contained in $\mc{D}$.  In fact, $\mc{E}$ is the smallest closed disc 
containing all the points $y_1, \ldots, y_{h+1}$.

Let $\mc{E}^{st}$ be the model of $\mc{E}$
whose special fiber is $\ol{V}$.  Since the points $y_1, \ldots, y_{h+1}$ are permuted by $G$, we have that $G$ acts on $\mc{E}^{st}$.  
Furthermore, since $\sigma \in G$ fixes each $y_i$, it follows from the stability of $\ol{Y}$ 
that $\sigma$ must fix the special fiber $\ol{V}$ of $\mc{E}^{st}$ pointwise.

Let $\mc{D}' = \mc{D}/\langle \sigma \rangle$, that is,
$$\mc{D}' = \Spec R[[U]]^{\langle \sigma \rangle} \cong \Spec R[[T]]$$ for some parameter $T$.  Then $\mc{D}'$ is an
open unit disc with generic fiber $\mc{D}'_K$.  Let $z_1, \ldots, z_{h+1}$ be the images of $y_1, \ldots, y_{h+1}$ under the canonical map 
$\mc{D} \to \mc{D}'$.  As above, we let $Z^{st}$ be the stable model of $\proj^1_K$ corresponding to the marked disc $(\mc{D}'; z_1, \ldots, z_{h+1})$,
with special fiber $\ol{Z}$.  If $\ol{\infty}'$ is the specialization of $\infty$ to $\ol{Z}$ and $\ol{W} = \ol{Z} \backslash \{\ol{\infty'}\}$, then we define
$\mc{E}'$ and $(\mc{E}')^{st}$ in analogy to $\mc{E}$, $\mc{E}^{st}$ in the previous paragraph.  We have $\mc{E}' = \mc{E}/\langle \sigma \rangle$ and 
$(\mc{E}')^{st} = \mc{E}^{st}/\langle \sigma \rangle$.  Note that the map $\mc{E}^{st} \to (\mc{E}')^{st}$ reduces to a purely inseparable map
$\ol{V} \to \ol{W}$ on the special fiber, as $\sigma$ fixes $\ol{V}$ pointwise.

We build the \emph{Hurwitz tree} corresponding to the $G$-action on $\mc{D}$ as follows (note that we will not explicitly define a Hurwitz tree until
the next section).  Consider the \emph{dual graph} $\Gamma$ of $\ol{Z}$, whose edges $E(\Gamma)$ and vertices $V(\Gamma)$ correspond to the 
irreducible components and nodes of $\ol{Z}$, respectively.  An edge connects two vertices if the corresponding node is the intersection of the two
corresponding components.   We append another vertex and edge $v_0$ and $e_0$ so that $e_0$ connects $v_0$ to the vertex corresponding to the 
component containing $\ol{\infty}'$. 

If $e \neq e_0$ is in $E(\Gamma)$, then $\epsilon_e$ is defined to be the thickness of the corresponding node of $\ol{Y}$ (\S\ref{Sdiscs}).  We set
$\epsilon_{e_0}$ to be the thickness of the open annulus $\mc{D} \backslash \mc{E}$. 
If $v \neq v_0$ is in $V(\Gamma)$, then let $\eta'$ be the generic point of the component $\ol{S}_v$ of $\ol{Z}$ corresponding to $v$, 
and let $\eta$ be the point of $\ol{Y}$ lying above $\eta'$.  
Then the extension $\hat{\mc{O}}_{\ol{Y}^{st}, \eta}/\hat{\mc{O}}_{\ol{Z}^{st}, \eta'}$ gives an inseparable extension on residue fields and 
satisfies the conditions of Proposition \ref{Pmupreduction}.  We set $\delta_v$ to 
be the valuation of the different of this extension (Proposition \ref{Pmupreduction}) and $\omega_v$ to be the corresponding deformation datum
(\S\ref{Sdefdata}).  We have that $\omega_v$ is a meromorphic differential form on $\ol{S}_v$.  We set $\delta_{v_0} = 0$, and we do not define
$\omega_{v_0}$.  

Lastly, we note that if we choose a subgroup $C \cong \ints/m \subseteq G$, then $C$ acts on $\mc{D}'$, $\mc{E}'$, and $(\mc{E}')^{st}$. 
So $C$ acts on $\ol{Z}$, as well as on $\Gamma$.  Also, if $c \in C$, then $c^*(\omega_{c(v)})$ is a differential form on the component 
corresponding to $v$.

\begin{definition}\label{Dhurwitz}
The data consisting of the group $C \subseteq G$, the curve $\ol{Z}$, the marking $\ol{\infty}'$, the specializations
$\ol{z}_1, \ldots, \ol{z}_{h+1}$ of $z_1, \ldots, z_{h+1}$ to $\ol{Z}$, the graph $\Gamma$, 
the differential forms $\omega_v$ for $v \in V(\Gamma) \backslash \{v_0\}$, the numbers $\delta_v$ for $v \in V(\Gamma)$ and $\epsilon_e$ for 
$e \in E(\Gamma)$, and the $C$-action on $\ol{Z}$, comprise the \emph{Hurwitz tree} of the $G$-action
on $\mc{D}$.  
\end{definition}

\subsubsection{Hurwitz trees in general.}\label{Shurwitzgeneral}
The data of Definition \ref{Dhurwitz} satisfy many compatibilities.  We will define a general Hurwitz tree to be a collection of data in the form 
above satisfying these compatibilities.

\begin{prop}\label{Paction2hurwitz}
The data in the Hurwitz tree of Definition \ref{Dhurwitz} satisfy the following properties:

\begin{enumerate}[(i)]
\item $\ol{Z}$ has genus $0$ and is stably marked by $\ol{\infty}, \ol{z}_1, \ldots, \ol{z}_{h+1}$ (i.e., each irreducible component contains at 
least three points that are either marked or singular).
\item For each $v \in V(\Gamma)$, we have that $\delta_v$ is rational and satisfies $0 \leq \delta_v \leq 1$, with $\delta_v = 0$ iff $v = v_0$.
\item For each $v \in V(\Gamma) \backslash \{v_0\}$, the divisor of $\omega_v$ is supported at the 
marked points and the singular points of $\ol{Z}$.  In particular, $\omega_v$ has simple
poles at any of the points $\ol{z}_1, \ldots, \ol{z}_{n+1}$ on the component of $\ol{Z}$ corresponding to $v$.
\item If $\delta_v = 1$, then $\omega_v$ is logarithmic.  Otherwise, $\omega_v$ is exact.
\item If $e \in E(\Gamma)$, then $\epsilon_e$ is a positive rational number.
\item If a node $\ol{z}$ of $\ol{Z}$ lies on two components $\ol{S}'$ and $\ol{S}''$ corresponding to vertices $v'$ and $v''$ of $\Gamma$, then
$\ord_{\ol{z}}\omega_{v'} + \ord_{\ol{z}}\omega_{v''} = -2$.
\item In the situation of (vi), if $v' \neq v_0$ and the node $\ol{z}$ corresponds to an edge $e \in E(\Gamma)$, then 
$$\delta_{v'} - \delta_{v''} = (p-1) \epsilon_e (\ord_{\ol{z}}\omega_{v'} + 1).$$
\item The action of $C$ fixes $\ol{\infty}'$ and permutes the points $\ol{z}_1 \ldots, \ol{z}_{h+1}$.
\item If $C_{\ol{z}} \subseteq C$ is the stabilizer of a node $\ol{z}$ of $\ol{Z}$, then the characters describing the action of $C_{\ol{z}}$ on the
tangent spaces of the two components of $\ol{Z}$ intersecting at $\ol{z}$ are inverse to each other. 
\item If $\chi$ is the conjugation character of $C$ on $\ints/p \subseteq G$ (\S\ref{Sdefdata}), then 
$c^*\omega_{c(v)} = \chi(c)\omega_v$ for all $c$, $v$.
\end{enumerate}
\end{prop}

\begin{proof}
Parts (i), (iv), and (v) follow from the definitions.  Part (ii) follows from Proposition \ref{Pmupreduction}.  Part (iii) follows from 
\cite[Prop.\ 1.7]{We:br}.  Parts (vi) and (vii) follow from \cite[Ch.\ 5]{He:ht}.  
Parts (viii) and (ix) follow because the Hurwitz tree comes from an action of $G$.  For a proof of part (x), see \cite[Prop.\ 3.4]{BW:ll}.
\end{proof}

\begin{definition}\label{Dgeneralhurwitz}
Suppose $\ol{Z}$ is a semistable curve over $k$ of genus $0$ with smooth distinct marked points $\ol{\infty}', \ol{z}_1, \ldots, \ol{z}_{h+1}$.
Let $\Gamma$ be the dual graph of $\ol{Z}$, with an extra vertex $v_0$ (called the \emph{root vertex}) 
and an extra edge $e_0$ connecting $v_0$ to the vertex corresponding
to the component containing $\ol{\infty}'$.  Suppose that each vertex $v \in V(\Gamma)$ has an
associated meromorphic differential form $\omega_v$ on the corresponding component (unless $v = v_0$) and an associated rational number 
$\delta_v$.  Suppose that each edge $e \in E(\Gamma)$ has an associated thickness $\epsilon_e$.  Lastly, suppose that the group $C \cong \ints/m$
acts on $\ol{Z}$ (thus on $\Gamma$), and has an injective character $\chi: C \to \FF_p^{\times}$ such that
$c^*\omega_{c(v)} = \chi(c)\omega_v$.  If this data satisfies parts (i)---(x) of 
Proposition \ref{Paction2hurwitz}, then it is called a \emph{Hurwitz tree of type $(C, \chi)$}.  
The \emph{conductor} of this Hurwitz tree is defined to be $h$.
We denote the Hurwitz tree by $(\ol{Z}; \ol{\infty}'; \ol{z}_1, \ldots, \ol{z}_{h + 1}; \omega_v; \delta_v; \epsilon_e)$
\end{definition}

Note that, as was mentioned at the beginning of \S\ref{Shurwitzaction}, the conductor of a $G$-action on the open unit disc at the boundary is 
equal to the conductor of its associated Hurwitz tree.  \\

\subsubsection{Lifting Hurwitz trees.}\label{Shurwitzlifting}
The main result of this section is the following:

\begin{prop}[\cite{BW:ll}, Theorem 3.6]\label{Phurwitz2action} 
Suppose we are given a Hurwitz tree $\mc{T} = (\ol{Z}; \ol{\infty}'; \ol{z}_1, \ldots, \ol{z}_{h + 1}; \omega_v; \delta_v; \epsilon_e)$
of type $(C \cong \ints/m, \chi)$ with conductor $h > 0$, and $\chi$ an injective character 
$C \to \FF_p^{\times}$.  If $G \cong \ints/p \rtimes_{\chi} C$, then there is a $G$-action on the open unit disc $\mc{D} = \Spec R[[U]]$, with
no inertia above a uniformizer of $R$, 
whose associated Hurwitz tree (Definition \ref{Dhurwitz}) is $\mc{T}$.  In particular, the 
conductor of the action of $G$ at the boundary $\del D = \Spec R[[U]]\{U^{-1}\}$ is $h$. 
\end{prop}

\begin{proof}[Very rough sketch of proof.]
The edges $e \in E(\Gamma)$ of the graph $\Gamma$ of $\mc{T}$ correspond to open annuli $\mc{A}_e$, 
whereas the vertices $v \in V(\Gamma)$ (other than the root vertex) correspond to (possibly punctured) closed discs $\mc{U}_v$, 
each $\mc{U}_v$ being a closed disc with $r_v - 1$ open discs of the same radius removed, where $r_v$ is the number of edges incident to $v$ 
(see \S\ref{Ssemistable}). 
These annuli and punctured discs fit together to form the open unit disc.  Furthermore, for each $\mc{U}_v$, we can use $\delta_v$ and 
$\omega_v$ to construct a $\ints/p$-cover of
$\mc{U}_v$.  Because of the compatibility properties of the $\delta_v$, $\omega_v$, and $\epsilon_e$ in the definition of Hurwitz tree, 
we can glue in $\ints/p$-covers of the annuli $\mc{A}_e$ to form an open unit disc.  This disc has a $\ints/p$-action, and it
turns out it even has a faithful $G$-action that satisfies the requirements of the proposition.
\end{proof}

\begin{proof}[More detailed version of proof.]
Maintain the notations from the sketch above.  Let $Z_R$ be a model of $\proj^1_K$ with special fiber $\ol{Z}$ 
(i.e., a flat $R$-scheme whose special fiber is $\ol{Z}$ and whose generic fiber is $\proj^1_K$) such that the specialization of $\infty$ is 
$\ol{\infty}'$.  Write $\mc{Z}$ for the formal completion 
of $Z_R$ at $\ol{Z}$.  Now, if $v \in V(\Gamma)$ is a vertex other than the root vertex, then the irreducible component $\ol{S}_v$ of $\ol{Z}$ 
corresponding to $v$ contains $r_v$ points $\ol{s}_1, \ldots, \ol{s}_{r_v}$
from among $\ol{\infty}'$ and the singular points of $\ol{Z}$.  These correspond to the edges of $\Gamma$ incident to $v$.  
If $\ol{U}_v \subseteq \ol{S}_v$ is the 
complement of these points, then the formal subscheme $\mc{U}_v \subseteq \mc{Z}$ that lifts $\ol{U}_v$ (\S\ref{Sformal}) is a projective line with $r_v$ 
open discs $\mc{D}'_{v, 1}, \ldots, \mc{D}'_{v, r_v}$ corresponding to $\ol{w}_1, \ldots, \ol{w}_{r_v}$ removed 
(alternatively, we can view $\mc{U}_v$ as a closed disc with $r_v - 1$ open discs removed, see \S\ref{Ssemistable}).  
Write $\mc{U}_v = \Spf A_v$.  In \cite[\S3.4]{BW:ll}, the data $\delta_v$ and $\omega_v$ are
used to construct a $\ints/p$-extension $B_v/A_v$, where $B_v$ is abstractly isomorphic to $A_v$.  Thus, 
the formal scheme $\mc{V}_v := \Spf B_v$ can
also be viewed as a formal projective line $\mc{Y}$ with $r_v$ open discs $\mc{D}_{v, 1}, \ldots, \mc{D}_{v, r_v}$ removed, with each 
$\mc{D}_{v, i}$ lying above the respective $\mc{D}'_{v, i}$.   

If $C_v \subseteq C$ is the stabilizer of $v$, then the action of $\ints/p$ on $B_v$ extends to an action of 
$G_v := \ints/p \rtimes C_v \subseteq G$ on $B_v$, thus on $\mc{V}_v$.   Each $\mc{D}_{v,i}$ has a boundary $\del \mc{D}_{v,i}$.  If $C_{v, i} 
\subseteq C_v$ is the stabilizer of the edge of $\Gamma$ corresponding to $\mc{D}_{v,i}'$, then $G_{v, i} := \ints/p \rtimes C_{v, i}$ acts on 
$\del \mc{D}_{v,i}$.
This gives rise to a $G_{v, i}$-cover $\del \mc{D}_{v,i} \to \del \mc{D}_{v,i}/G_{v,i}$.  
It is shown (\cite[Lemma 3.3(iii), Prop.\ 3.9]{BW:ll}) that this extension is of type 
$(\delta_v, h_{v, i} := -\ord_{\ol{w}_i}(\omega_v) - 1, \alpha)$, where $\alpha: C_{v, i} \to (\FF_p)^{\times}$ is the unique character
such that $\alpha^{-h_{v,i}} = \chi |_{C_{v, i}}$ (note that $h_{v,i}$ is called $h_i$ in \cite[Prop.\ 3.9]{BW:ll}.  There is a sign
error in that paper which explains why our $h_{v,i}$ is their $h_i$).

Now, suppose $e \in E(\Gamma)$ is an edge with stabilizer $C_e$.  Consider the two vertices $v_1$ and $v_2$ incident to $e$.  
Assume for the moment that neither $v_1$ nor $v_2$ is the root vertex.  Then we have seen that there is a 
$G_e := \ints/p \rtimes C_e$-cover of boundaries of open discs associated to each of $v_1$ and $v_2$, namely, the covers
$\del \mc{D}_{v_1, i} \to \del \mc{D}_{v_1, i}/G_e$, and $\del \mc{D}_{v_2, j} \to \del \mc{D}_{v_2, j}/G_e$ for the correct $i$ and $j$ 
corresponding to $e$.  Suppose
these covers are of type $(\delta_1, h_1, \alpha_1)$ and $(\delta_2, h_2, \alpha_2)$, respectively.
By (vi) of Proposition \ref{Paction2hurwitz}, we have $h_2 = -h_1$.  This gives us that
$\alpha_2 = \alpha_1^{-1}$.  By (ii) and (vii) of Proposition \ref{Paction2hurwitz}, we have $0 \leq \delta_1, \delta_2 \leq 1$ and
$$\delta_1 - \delta_2 = (p-1)\epsilon_e (-h_1) = (p-1)\epsilon_e h_2.$$  In this situation, \cite[\S3.5]{BW:ll} exhibits an open annulus 
$\mc{A}_e$ with thickness $\epsilon_e$ and $G_e$-action such that one boundary is of type
($\delta_1$, $h_1$, $\alpha_1$) whereas the other is of type ($\delta_1 + (p-1)\epsilon_e h_1$, $-h_1$, $\alpha_1^{-1}$).  By Lemma 
\ref{Lboundarystructure}, one can identify the boundaries of $\mc{A}_e$ with $\del \mc{D}_{v_1, i}$ and $\del \mc{D}_{v_2, j}$ in a 
$G_e$-equivariant way, and thus we can glue $\mc{A}_e$ to $\mc{V}_{v_1}$ and $\mc{V}_{v_2}$ while respecting the $G_e$-action.

Alternatively, 
If $e$ is the edge $e_0$, incident to $v_1$ and the root vertex $v_0$, then \cite[\S3.5]{BW:ll} constructs an annulus $\mc{A}_{e_0}$ of thickness 
$\epsilon_{e_0}$ with one boundary of type $(\delta_1, h_1, \alpha_1)$ and the other of type $(0, -h_1, \alpha_1^{-1})$.  
Again, this annulus can be glued along its first boundary to $\mc{V}_{v_1}$ while respecting the $G_e$-action.  We know that $-h_1$ is just 
$\ord_{\ol{\infty}'}(\omega_{v_1}) + 1$, which by \cite[Lemma 3.3(i)]{BW:ll} is equal to $h$.

By gluing the annuli $\mc{A}_e$ for all $e \in E(\Gamma)$ to the $\mc{V}_v$ for all $v \ne v_0 \in V(\Gamma)$ as above, we obtain an 
open disc $\mc{D}$.  The $G$-action on $\mc{D}$ 
is given by having $G$ permute the $\mc{V}_v$ and $\mc{A}_e$ just as it permutes $V(\Gamma)$ and $E(\Gamma)$, 
and having the stabilizers $G_v$ and $G_e$ of each vertex and edge act as in the construction.  One checks that the $G$-action gives rise to
the original Hurwitz tree $\mc{T}$.
\end{proof}

\begin{corollary}\label{Churwitz2action}
If $\mc{D} \cong \Spec R[[U]]$ is the open disc with $G$-action
corresponding to the Hurwitz tree in Proposition \ref{Phurwitz2action}, then the reduction of the
extension $R[[U]]/R[[U]]^G$ to characteristic $p$ has unique positive lower jump $h$.
\end{corollary}

\begin{proof}
By Proposition \ref{Phurwitz2action}, the group $G$ acts on $R[[U]]\{U^{-1}\}$ with conductor $h$.  The reduction of 
$R[[U]]/R[[U]]^G$ to characteristic $p$ is a $G$-extension whose associated extension of fraction fields is the reduction of
$R[[U]]\{U^{-1}\}/R[[U]]\{U^{-1}\}^G$.  By the definition of conductor and the lower numbering, the reduction of $R[[U]]/R[[U]]^G$ has unique
lower jump equal to the conductor of $R[[U]]\{U^{-1}\}/R[[U]]\{U^{-1}\}^G$, which is $h$.
\end{proof}

\subsubsection{Constructing Hurwitz trees.}\label{Shurwitzconstruction}
In order to complete the proof of Theorem \ref{Tzpzm}, we need only show:

\begin{prop}\label{Phurwitzconstruction}
If $C \cong \ints/m$, if $\chi: C \to \FF_p^{\times}$ is an injective character, and if $h > 1$ satisfies $h \equiv -1 \pmod{m}$, then
there exists a Hurwitz tree of type $(C, \chi)$ with conductor $h$.  
\end{prop}

\begin{proof}
If $h < p$, we give the construction from \cite[Proof of Prop.\ 1.4]{BWZ:dd}.  
We take $\ol{Z}$ to be \emph{smooth}, so that the graph $\Gamma$ will only have
two vertices $v_0$ and $v_1$, one edge $e_0$ incident to both, and $v_0$ is the root vertex.  
Set $\delta_{v_0} = 0$, $\delta_{v_1} = 1$, and $\epsilon_{e_0} = \frac{1}{h(p-1)}$.  Write $h = rm - 1$.
Consider $z_1, \ldots, z_r \in \FF_p^{\times}$ such that the elements
$$z_{i,j} := \zeta_m^j z_i, \ \ 1 \leq i \leq r, \ 0 \leq j \leq m-1$$ are pairwise distinct (note: $h \ne p-1$).  
Here $\zeta_m$ is an $m$th root of unity, which we think of as an element of $\FF_p^{\times}$.  Take $\infty$ and the $z_{i,j}$ to be the
marked points.  Then set $$\omega_{v_1} = \frac{dz}{\prod_{i=1}^r (z^m -z_i^m)}.$$  One shows that this is logarithmic (\cite[Lemma 1.5]{BWZ:dd}).  
It has a zero at $\infty$ of order $rm-2 = h-1$.  Lastly, if $c$ is a generator of $\ints/m$, take the action of $c$ on $\ol{Z}$ to be given by 
$c^*z = \zeta_m z$.  It is then an easy exercise to verify properties (i)---(x) of Proposition \ref{Paction2hurwitz}.

If $h > p$, then we cannot take $\ol{Z}$ to be smooth, and the construction is more complicated.  See \cite[Proof of Theorem 2.1]{BWZ:dd} and
\cite[Theorem 4.3]{BW:ll}.
\end{proof}

\begin{remark}\label{Rbighurwitz}
The methods used to prove Theorem \ref{Tzpzm} work equally well to show that all $\ints/p$-extensions of $k[[t]]$ lift to characteristic zero
(this is the subject of Henrio's thesis \cite{He:ht}).  Of course, one would like to extend these methods to lift $\ints/p^n$- and
$\ints/p^n \rtimes \ints/m$-extensions of $k[[t]]$ for which the KGB obstruction vanishes.  Indeed, there is a notion of Hurwitz tree for
actions of $\ints/p^n \rtimes \ints/m$ on the open unit disc (see \cite{BW:ac}).  However, there is currently no good analogue of Lemma
\ref{Lboundarystructure} for groups with cyclic $p$-Sylow subgroups of order greater than $p$.  Without such an analogue, the gluing that is
essential for the proof of Proposition \ref{Phurwitz2action} cannot be performed.
\end{remark}

\appendix
\section{Non-archimedean geometry}\label{Sgeometry}
We briefly discuss the constructions and objects in non-archimedean geometry that are used in this paper, taking a rather na\"{i}ve perspective 
throughout.  
For a thorough introduction to formal and rigid geometry, see e.g. \cite{BGR:na}, or for a quick overview see \cite{Ga:gr} and \cite{He:dc}.
Throughout this appendix, let $R/W(k)$ be finite, and let $K = \Frac(R)$.

\subsection{Discs and annuli}\label{Sdiscs}
We call the scheme $\Spec R[[T]]$ the $\emph{open unit disc}$
(by abuse of language, the term ``open unit disc" is used regardless of $R$).  
For any algebraically closed field extension $K'/K$ where $K'$ has an absolute value extending that of $K$, the 
$K'$-points of $\Spec R[[T]]$ correspond to the elements $b \in K'$ with $|b| < 1$, by plugging in $T = b$ 
(these are the values at which all power series in $R[[T]]$ converge).  By abuse of language, we will call this set of $K'$-points the open
unit disc as well (we will do the same for the closed disc and annuli described below).

The scheme $\Spec R\{T\}$ is the $\emph{closed unit disc}$ (recall that $R\{T\}$ is the subring of $R[[T]]$ consisting of power series for which the
coefficients tend to $0$, see \S\ref{Snotation}).
For $K'/K$ as above, the $K'$-points of $\Spec R\{T\}$ correspond to the elements $b \in K'$ with $|b| \leq 1$, by plugging in $T = b$. 
The $\emph{boundary}$ of the open (or closed) unit disc is the scheme $\Spec R[[T^{-1}]]\{T\}$.  Note that this scheme has no $K$-points. 
The ring $R[[T^{-1}]]\{T\}$ is a complete DVR with residue field $k((t^{-1}))$.

If $a \in K$ such that $|a| = r$, then $\Spec R[[a^{-1}T]]$ (resp.\ $\Spec R\{a^{-1}T\}$) is the open (resp.\ closed) disc of radius $r$.  
This is isomorphic to the open (resp.\ closed) unit disc under the map $T \mapsto a^{-1}T$.

The \emph{open annulus of thickness $\epsilon$} is $\Spec R[[T, U]]/(TU - a)$, where $a \in R$ is such that $v(a) = \epsilon$ 
(recall that we always set $v(p) = 1$).  
For $K'/K$ as above, the $K'$-points of $\Spec R[[T, U]]/(TU - a)$ correspond to those $b \in K'$ with $|a| = p^{-\epsilon} < |b| < 1$ (or, stated in terms 
of the
valuation on $K'$, that $0 < v(b) < \epsilon$), by plugging in $T = b$.  The open annulus has two boundaries, one given by $\Spec R[[T]]\{T^{-1}\}$ 
and one given by $\Spec R[[U]]\{U^{-1}\}$.  Note that two open annuli over $R$ are isomorphic if and only if they have the same thickness.

\subsection{Semistable models of curves}\label{Ssemistable}
Let $X_R \to \Spec R$ be a semistable curve with smooth generic fiber $X$ and special fiber $\ol{X}$.  
If $\ol{x} \in \ol{X}$, then $\ol{x}$ is either a smooth point or a node.
If $\ol{x}$ is smooth, then $\Spec \hat{\mc{O}}_{X_R, \ol{x}}$ is isomorphic to an open unit disc, whereas if $\ol{x}$ is a node, 
then $\Spec \hat{\mc{O}}_{X_R, \ol{x}}$
is isomorphic to an annulus of some thickness $\epsilon$, which is an intrinsic property of the singularity $\ol{x} \in X_R$
(see, e.g., \cite[\S2.1.1]{Ra:sp}).  

If $X$ is a projective line, we can give more detail.  First assume that $X_R$ is smooth.  
Then there are (many) elements $T$ in the function field of $X$ such that $K(T) \cong K(X)$ and the local
ring at the generic point of the special fiber of $X_R$ is the valuation ring of $K(T)$ corresponding to the
Gauss valuation.  We say that $T$ is a \emph{coordinate} of $X_R$ (note that there can be many different coordinates for a given model).  
If $T$ is a coordinate of $X_R$, if $a, b \in R$, and if $c \in K$, 
then $T=a$ and $T=b$ coalesce on the special fiber iff $v(a-b) > 0$, and $T= \infty$ and $T=c$ coalesce on the special fiber iff $c \notin R$. 

Conversely, if $T$ is any rational function on $X$ such that $K(T) \cong K(X)$, there is a smooth
model $X_R$ of $X$ such that $T$ is a coordinate of $X_R$.

Now, drop the assumption that $X_R$ is smooth.
Then $\ol{X}$ is a tree-like configuration of $\proj^1_k$'s.
Each irreducible component $\ol{W}$ of the special fiber $\ol{X}$ of $X_R$ corresponds to
a smooth model of $X$, and thus to (many) coordinates $T$.  Such a $T$ is called a \emph{coordinate on $\ol{W}$} in this case.

Let $X_R$ be a semistable model for $X = \proj^1_K$.  Let $\ol{K}$ be an algebraic closure of $K$.  Suppose $\ol{x}$ is a smooth point
of $\ol{X}$ on the irreducible component $\ol{W}$.  Let $T$ be a coordinate on $\ol{W}$ such
that $T = 0$ specializes to $\ol{x}$.  The complete local ring of $\ol{x}$ in $X_R$ is isomorphic to $R[[T]]$. 
The set of points of $X(\ol{K})$ that specialize to $\ol{x}$ is the open unit disc $v(T) > 0$.
  
Now, let $\ol{x}$ be a nodal point of $\ol{X}$, lying on an irreducible component $\ol{W}$. Suppose $T$
is a coordinate on $\ol{W}$ such that $T=0$ specializes to the connected component of $\ol{X} \backslash \{\ol{x}\}$ not containing 
$\ol{W} \backslash \{\ol{x}\}$.  Then the complete local ring of $\ol{x}$ 
in $X_R$ is isomorphic to $R[[T,U]]/(TU - p^{\epsilon})$ where $\epsilon$ is the thickness of the annulus corresponding to $\ol{x}$.
The set of points of $X(\ol{K})$ that specialize to $\ol{x} \in \ol{X}$ is the open annulus given by $0 < v(T) < \epsilon$. 

Suppose $T$ is a coordinate on a component $\ol{W}$ of $\ol{X}$, and suppose $T = \infty$ specializes away from $\ol{W}$.  
Let $\ol{U}$ be the subset of $\ol{W}$ consisting of smooth points of $\ol{X}$.  Then the set of points of 
$X(\ol{K})$ that specialize to $\ol{U}$ is a punctured closed unit disc, that is, a closed disc with an open unit disc removed for each
point of $(\ol{X} \backslash \{\ol{\infty}\}) \backslash \ol{U}$.

\subsection{Stable models and marked discs}\label{Smarkeddiscs}
Let $X = \proj^1_K$, and suppose we are given $x_1, \ldots, x_r \in X(K)$, with $r \geq 3$.
Then there is a unique semistable model $X^{st}$ for $X$ over $R$ such that the specializations of $x_1, \ldots, x_r$ to the special fiber 
$\ol{X}$ of $X^{st}$ are pairwise distinct, and such that each irreducible component of $\ol{X}$ contains at least three points that 
are either singular or specializations of an $x_i$.  The $R$-curve $X^{st}$ is called the \emph{stable model} of the marked curve 
$(X; x_1, \ldots, x_r)$.

Consider the open unit disc $\mc{D} := \Spec R[[T]]$, and suppose we are given $x_1, \ldots, x_r$ in $\mc{D}(K)$, with $r \geq 2$.
From \S\ref{Sdiscs}, we can think of $x_1, \ldots, x_r$ as elements of the maximal ideal of $R$.  
Let $X^{st}$ be the stable model of the marked curve $(\proj^1_K; \infty, x_1, \ldots, x_r)$, with special fiber $\ol{X}$ (thinking of
$\proj^1(K)$ as $K \cup \{\infty\}$).  We will call $X^{st}$ the \emph{stable model of $\proj^1_K$ corresponding to the marked disc 
$(\mc{D}; x_1, \ldots, x_r)$}.  Let $\ol{\infty}$ be the specialization of $\infty$ to $\ol{X}$.  
Note that the set of points of $\proj^1(K)$ that specialize away from $\ol{\infty}$ 
form a closed disc.  In particular, each point in this closed disc is an element of the maximal ideal of $R$ (although some elements of the
maximal ideal of $R$ might specialize to $\ol{\infty}$).

\begin{example}\label{Eseparateddisc}
Let $R = W(\ol{\FF_5})[a]/(a^2 - 5)$, and consider the marked open unit disc $(\mc{D}; 0, 5, 10, 25, 5a, 5 + 5a)$.
Then the stable model of $X = \proj^1_K$ corresponding to our marked disc has a special fiber $\ol{X}$
with four irreducible components $\ol{X}_1$, $\ol{X}_2$, $\ol{X}_3$, and $\ol{X}_4$, as in Figure \ref{Fstablemodel}
(the specializations of the marked points are marked with overlines in Figure \ref{Fstablemodel}).  

\begin{figure}[htp]
\centering
\includegraphics[totalheight=.4\textheight, width=.85\textwidth]{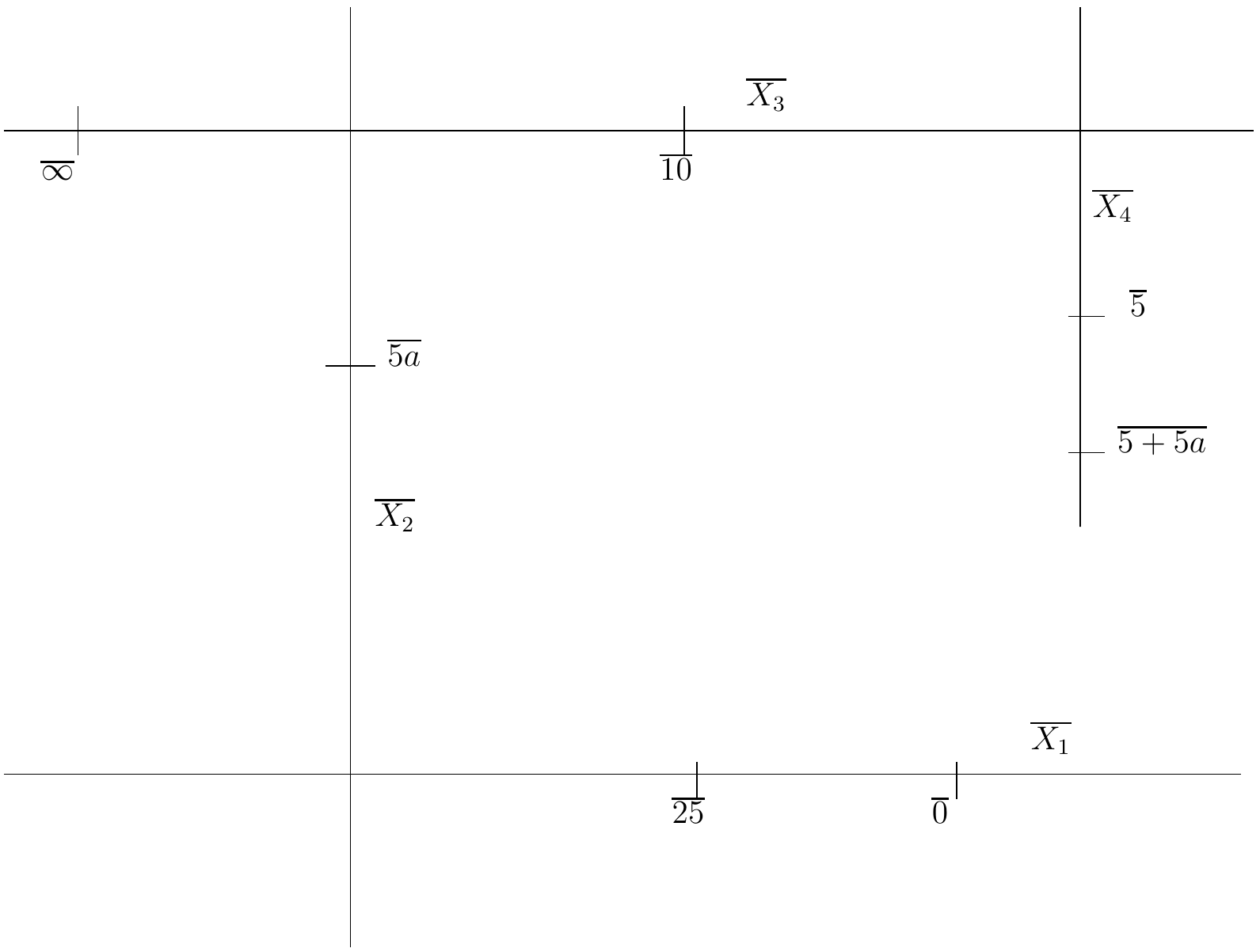}
\caption{The special fiber $\ol{X}$}\label{Fstablemodel}
\end{figure}

The following table shows where the $\ol{K}$-points of $X$ specialize.\\

\begin{tabular}{l|l}
Subscheme $\ol{V}$ of $\ol{X}$ & Points of $X(K)$ that specialize to $\ol{V}$\\
\hline 
$\ol{X}_1 \backslash \ol{X}_2$ & $\{ x |\ v(x) \geq 2 \}$ \\
$\ol{X}_1 \cap \ol{X}_2$ & $\{ x |\ \frac{3}{2} < v(x) < 2 \}$ \\
$\ol{X}_2 \backslash (\ol{X}_1 \cup \ol{X}_3)$ & $\{ x |\ v(x) = \frac{3}{2} \}$ \\
$\ol{X}_2 \cap \ol{X}_3$ & $\{ x |\ 1 < v(x) < \frac{3}{2} \}$ \\
$\ol{X}_3 \backslash (\ol{X}_2 \cup \ol{X}_4 \cup \{\ol{\infty}\})$ & $\{ x |\ v(x) = 1 \wedge v(x-5) = 1 \}$ \\ 
$\ol{X}_3 \cap \ol{X}_4$ & $\{x |\ 1 < v(x-5) < \frac{3}{2} \}$\\
$\ol{X}_4 \backslash \ol{X}_3$ & $\{ x |\ v(x - 5) \geq \frac{3}{2} \}$ \\
$\ol{\infty}$ & $\{ x |\ v(x) < 1 \}$\\
\end{tabular}
\end{example}

\subsection{Formal schemes}\label{Sformal}
Corresponding to each disc or annulus from \S\ref{Ssemistable}, there is a \emph{formal} disc or annulus.  In particular, if a disc or annulus is 
given by $\Spec A$, then the corresponding formal disc or annulus is $\Spf A$.  In fact, if $\mc{X}$ is the formal completion of $X_R$ 
at its special fiber (the ``formal projective line"), then any subscheme $\ol{V} \subseteq \ol{X}$ lifts to a (canonical) formal subscheme $\mc{V} \subseteq 
\mc{X}$, regardless of the genus of $X$.  For instance, take any (Zariski) open subscheme $V_R$ of $X_R$ such that $\ol{V}$ is a closed subscheme 
of $V_R$.  Then $\mc{V}$ is the formal completion of $V_R$ at $\ol{V}$, viewed as a formal subscheme of $\mc{X}$ via the inclusion $V_R \to X_R$.

\section{Depth and deformation data}\label{Sstablegraph}
\subsection{Depth}\label{Sdepth}
Maintain the notation $R$ and $K$ from Appendix \ref{Sgeometry}.  For any $R$-scheme $X$ in this appendix, write $X_K$ for its generic fiber
and $X_k$ for its special fiber.  Likewise, for an $R$-algebra $A$, write $A_K$ for $A \otimes_R K$ and $A_k$ for $A \otimes_R k$.  We assume
throughout that $R$ contains the $p$th roots of unity.

Given a $G$-extension $B/R[[T]]$, it is important in \S\ref{Sdepthandsep} 
to know how far its reduction is from being separable.  In particular, if the reduction is separable,
then $B/R[[T]]$ is at least a birational lift of the normalization of the reduction.  The most basic interesting case is when $G \cong \ints/p$.
We start off by stating a general structure theorem about $\ints/p$-extensions of certain $R$-algebras.

\begin{prop}[\cite{He:ht}, Ch.\ 5, Proposition 1.6]\label{Pmupreduction}
Let $X = \Spec A$ be a flat affine scheme over $R$, with relative dimension $\leq 1$ and integral fibers.
We suppose that $A$ is a factorial $R$-algebra that is complete with respect to the $\pi$-adic valuation (for $\pi$ a uniformizer of $R$).
Let $Y_K \to X_K$ be a non-trivial \'{e}tale $\ints/p$-cover, given by an equation $y^p = f$, where $f$ is invertible in $A_K$.
Let $Y$ be the normalization of $X$ in $K(Y_K)$; we suppose the special fiber of $Y$ is integral (in particular, reduced).
Let $\eta$ (resp.\ $\eta'$) be the generic point of the special fiber of $X$ (resp.\ $Y$).  
The local rings $\mc{O}_{X, \eta}$ and $\mc{O}_{Y, \eta'}$ are thus discrete
valuation rings with uniformizer $\pi$.  Write $\delta$ for the valuation of the different of $\mc{O}_{Y,
\eta'}/\mc{O}_{X, \eta}$.  We then have two cases, depending on the value of $\delta$ (which always
satisfies $0 \leq \delta \leq 1$):

\begin{itemize}
\item If $\delta = 1$, then $Y \cong \Spec B$, with $B = A[y]/(y^p - u)$, for
$u$ a unit in $A$, not congruent to a $p$th power modulo $\pi$, and unique up to
multiplication by a $p$th power in $A^{\times}$.  We say that the $\ints/p$-cover $Y_K \to X_K$ has
\emph{multiplicative reduction}.  \item If $0 \leq \delta < 1$, then $\delta = 1 - n(\frac{p-1}{e})$, where $n$ is
an integer such that $0 < n \leq e/(p-1)$.  Then $Y \cong \Spec B$, with $$B = \frac{A[w]}{(\frac{(\pi^n w+1)^p
-1}{\pi^{pn}} - u)},$$ for $u$ a unit of $A$, not congruent to a $p$th power modulo $\pi$.  Also, $u$ is unique in the
following sense:  If an element $u' \in A^{\times}$ could take the place of $u$, then there exists $v \in A$ such that 
$$\pi^{pn}u' + 1 = (\pi^{pn}u + 1)(\pi^nv + 1)^p.$$  
If $\delta > 0$ (resp.\ $\delta = 0$), we say that the $\ints/p$-cover $Y_K \to X_K$ has
\emph{additive reduction} (resp. \emph{\'{e}tale reduction}).
\end{itemize}
\end{prop}

Note that the map $Y_k \to X_k$ above is separable iff $\delta = 0$.  

\begin{definition}\label{Ddepth}
\begin{enumerate}[(i)]
\item If $\delta$ is as in Proposition \ref{Pmupreduction}, then the number $\frac{p}{p-1} \delta$ is called the \emph{depth} of the extension 
$B/A$ (or of the map $Y \to X$). 
\item If $X$ is as in Proposition \ref{Pmupreduction}, let $Y_K \to X_K$ be a $\ints/p^n$-cover that can be broken up into a tower of
$\ints/p$-covers $Y_K = T_n \to T_{n-1} \to \cdots \to T_1 \to T_0 = X_K$.  Assume that for all $i$, 
$1 \leq i \leq n$, the maps $T_{i+1} \to T_i$ and the normalization $A_i$ of $X$ in $T_i$ satisfy the properties of Proposition 
\ref{Pmupreduction} (with $A_i$ playing the role of $X$ and $T_{i+1} \to T_i$ playing the role of $Y_K \to X_K$).  
Then for each $T_i/T_{i-1}$, we can define $\delta_i$ as in Proposition \ref{Pmupreduction}.  If the normalization of $X$ in $K(Y_K)$ is $Y = \Spec B$, 
then
the depth of $B/A$ (or of $Y \to X$) is defined to be $$\left(\sum_{i=1}^{n-1} \delta_i\right) + \frac{p}{p-1} \delta_n.$$  This is called the \emph{effective
different} in \cite{Ob:fm1} and \cite{Ob:fm2}.
\end{enumerate}
\end{definition}

\subsection{Deformation data}\label{Sdefdata}
Let $Y \to X = \Spec A$ be a degree $p$ finite extension as in Proposition \ref{Pmupreduction}, let $\delta$ and $u$ be as in Proposition 
\ref{Pmupreduction}, let $\ol{u}$ be the reduction of $u$ to $A_k$, and assume $\delta > 0$.  Then the $\ints/p$-cover $Y \to X$ 
gives rise to a meromorphic differential 
form $\omega$ on $\Spec A_k$ as follows: if $Y_K \to X_K$ has multiplicative reduction, set $\omega = d\ol{u}/\ol{u}$.  If $Y_K \to X_K$ has
additive reduction, set $\omega = d\ol{u}$.  One verifies that $\omega$ is independent of the choice of $u$, and $\omega$ is called the 
\emph{deformation datum} corresponding to the map $Y \to X$.  See \cite[Ch.\ 5]{He:ht} for more details.

Furthermore, suppose that the natural $\ints/p$-action on $Y$ extends to a $\ints/p \rtimes \ints/m$ action on $Y$, with $p \nmid m$, which
descends to a $\ints/m$-action on $X$.  Let $\chi: \ints/m \to \FF_p^{\times}$ be the character given by $c\sigma c^{-1} = \sigma^{\chi(c)}$, for
$c \in \ints/m$ and $\sigma$ any generator of $\ints/p$.  Then, for any $c \in \ints/m$, we have $c^*\omega = \chi(c)\omega$ (this is mentioned in
\cite[p.\ 999]{We:br}, and proven in \cite[Construction 3.4]{Ob:vc}).  

\begin{remark}\label{Rdefdata}
The deformation datum $\omega$ is important in constructing Hurwitz trees, see \S\ref{Shurwitzaction}.  In particular, it is a characteristic
$p$ object that helps retain information that is lost when a branched cover is reduced from characteristic zero.  Deformation data are also extremely 
important in the proof of Theorem \ref{Tzpnlifting} by Wewers and the author, as they are used to construct the differential equations
mentioned in the proof of Proposition \ref{Psmoothing}.
\end{remark}
  
\begin{remark}\label{Rkato}
The depth and deformation datum of a map $Y \to X$ are the two components of Kato's differential Swan conductor, defined in \cite{Ka:sc}
(also see \cite[Ch.\ 1]{Br:rt}, especially \S1.4,  for an exposition).
\end{remark}

\section*{Acknowledgements}
I thank David Harbater and Stefan Wewers for useful conversations, as well as the referee for useful comments. 

\addcontentsline{toc}{section}{Bibliography}

\end{document}